\documentclass[11pt]{amsart}
\usepackage[margin=1in]{geometry}
\usepackage{amsmath, amssymb, amsfonts, mathtools}
\usepackage{enumitem}
\usepackage{graphicx}
\usepackage{asymptote}
\usepackage{caption}
\usepackage{dsfont} 
\usepackage{comment}

\usepackage{thmtools}

\newtheorem{theorem}{Theorem}[section]
\newtheorem{proposition}[theorem]{Proposition}

\newtheorem{conjecture}[theorem]{Conjecture}

\newtheorem{lemma}[theorem]{Lemma}

\theoremstyle{definition}

\usepackage{xcolor}
\usepackage{hyperref}

\definecolor{LightBlue}{rgb}{0,0.8,1} 

\hypersetup{colorlinks=true, citecolor=LightBlue, linkcolor=blue,urlcolor=blue} 

\usepackage[noabbrev,capitalise,nameinlink]{cleveref}
\crefname{conjecture}{Conjecture}{Conjectures}

\newcommand{\ZZ}{\mathbb{Z}}
\newcommand{\RR}{\mathbb{R}}

\newcommand{\FS}{\mathrm{FS}}
\newcommand{\SI}{\mathrm{SI}}

\DeclareMathOperator{\lcm}{lcm}

\DeclareMathOperator{\covol}{covol}
\DeclareMathOperator{\vol}{vol}

\title{Sets with few subset sums}
\author{Ruben Carpenter}
\address{Yale University, 421 Temple St, New Haven, CT 06511}
\email{ruben.carpenter@yale.edu}
\author{Colin Defant}
\address{Harvard University, 1 Oxford St, Cambridge, MA 02139}
\email{colindefant@gmail.com} 
\author{Noah Kravitz}
\address{St John's College, Oxford and Mathematical Institute, University of Oxford; St Giles', Oxford OX1 3JP, UK}
\email{noah.kravitz@maths.ox.ac.uk}

\begin{document}

\begin{abstract}
It is a classical fact that every $n$-element set of positive reals has at least $\binom{n+1}{2}+1$ distinct subset sums, with equality exactly for homogeneous arithmetic progressions (when $n\geq 4$).  We establish stability versions of this inverse theorem in two regimes.  First, for any parameter $M \leq n-4$, we precisely characterize the $n$-element sets of positive reals with at most $\binom{n+1}{2}+1+M$ subset sums.  Second, for any constant $C$, we provide a characterization, sharp up to constants, of the $n$-element sets of positive reals with at most $Cn^2$ distinct subset sums.  Along the way, we constrain (for any fixed $d \geq 2$) the structure of $n$-element subsets of $\mathbb{R}^d$ with $o(n^{d+1})$ subset sums. 
\end{abstract}

\maketitle

\section{Introduction}

\subsection{Inverse and stability theorems}
A central theme in additive combinatorics is that sets with small doubling have constrained structure.  It is a basic fact that if $A \subseteq \mathbb{R}$ is an $n$-element set, then the sumset $A+A$ has size at least $2n-1$.  The ``inverse theorem'' for this inequality is the statement that equality holds if and only if $A$ is an arithmetic progression.  Freiman's $3k-4$ Theorem provides a ``stability'' version of this inverse theorem, namely, a precise sense in which a set with doubling near the minimum must closely resemble an arithmetic progression.  The Freiman--Ruzsa Theorem provides a rougher structural description of sets of bounded doubling, namely, that any $n$-element set $A \subseteq \mathbb{R}$ with $|A+A|=O(n)$ is contained in a generalized arithmetic progression  (GAP) of rank $O(1)$ and size $O(n)$; this characterization is tight up to the dependence of the implicit constants.  In the present paper we study analogous problems for sets with few distinct subset sums.

For a finite subset $A$ of an abelian group $G$, let
$$\FS(A):=\left\{\sum_{x \in B}x: B \subseteq A\right\}$$
denote the set of \emph{subset sums} of $A$. (The notation ``$\FS$'', standing for ``finite sums'', originated in arithmetic Ramsey theory.)  In the most natural setting $G=\mathbb{R}$, an easy inductive argument (see~\cite[Theorem 3]{N95}) gives that if $A \subseteq \mathbb{R}_{>0}$ is an $n$-element set, then
$$|\FS(A)| \geq \binom{n+1}{2}+1.$$
One can further show (see~\cite[Theorem 5]{N95}) that when $n\geq 4$, equality occurs if and only if $A$ is a \emph{homogeneous arithmetic progression}, namely, a set of the form
$$A=\{a,2a, \ldots, na\}$$
for some $a>0$.  The goal of this paper is to establish stability versions of this inverse theorem.  We study two regimes, one with a more ``exact'' flavor (like Freiman's $3k-4$ Theorem) and one with a more ``asymptotic'' flavor (like the Freiman--Ruzsa Theorem).

This line of inquiry is closely related to inverse Littlewood--Offord theory.  Littlewood--Offord theory concerns the anticoncentration properties of the random variable $S:=\sum_{x \in B}x$, where $B$ is a uniformly random subset of a fixed $n$-element (multi)set $A \subseteq \mathbb{R}_{>0}$.  Inverse Littlewood--Offord theory, as introduced by Tao and Vu~\cite{TV09}, aims to describe the structure of the (multi)sets $A$ with $\max_{y \in \RR}\mathbb{P}[S=y]$ large; a typical result (see \cite[Theorem 2.5 and Corollary 2.7]{TV09}) says that if $\max_{y \in \RR}\mathbb{P}[S=y] \geq n^{-O(1)}$, then almost all of $A$ is contained in a generalized arithmetic progression of rank $O(1)$ and size $n^{O(1)}$.

Conceptually speaking, our main results describe the structure of the set $A$ when the random variable $S$ has \emph{small support}, while the work of Tao and Vu describes the structure of $A$ when the random variable $S$ has \emph{some popular value}.  The former hypothesis is of a stronger form than the latter, and we obtain the accordingly stronger conclusion that $A$ must resemble a  generalized arithmetic progression of rank $1$, rather than merely fairly small rank; this ``rank-collapse'' phenomenon (see \cref{prop:proj'} below) is one of the main takeaways of the paper.

\subsection{Main results}
We first study the regime where $|\FS(A)|$ only slightly exceeds $\binom{n+1}{2}+1$, as proposed by Nathanson in 1995~\cite{N95}.  For $c\in \RR^\times$, let $c \cdot A$ denote the dilation of $A$ by $c$.

\begin{theorem}\label{thm:linear}
Let $n\geq 4$ and $0 \leq M \leq n-4$ be integers.  Let $A \subseteq \mathbb{R}_{>0}$ be an $n$-element set.  Then $$|\FS(A)| \leq \binom{n+1}{2}+1+M$$
if and only if there is some $a>0$ such that $A \subseteq a\cdot \mathbb{Z}_{>0}$ and
$$\sum_{x \in a^{-1} \cdot A} x \leq \binom{n+1}{2}+M.$$
\end{theorem}
The latter condition in \cref{thm:linear} implies that
$$\{1,2, \ldots, n-M\} \subseteq a^{-1} \cdot A \subseteq \{1,2, \ldots, n+M\},$$
which, when $M$ is small, indicates a precise sense in which $A$ is ``close'' to being a homogeneous arithmetic progression.  The constraint $M \leq n-4$ is optimal because of examples like
$$A=\{1,3,4,\ldots, n+1\},$$
which has $|\FS(A)|=\binom{n+1}{2}+1+(n-3)$ and $\sum_{x \in A} x=\binom{n+1}{2}+n-1$.  See \cite{lev} for older results of a similar flavor; also see \cite{lev2,GHHLLP} and the references therein for thematically related work on subsequence sums of zero-sum-free sequences in finite abelian groups.

Our second main result concerns the regime where $|\FS(A)|=O(n^2)$.
\begin{theorem} \label{thm: main inverse thm}
Let $C>0$, and let $n$ be a positive integer.  If $A \subseteq \mathbb{R}_{>0}$ is an $n$-element set with $|\FS(A)| \leq Cn^2$, then there is a decomposition $A=A_1 \sqcup A_2$, where: 
\begin{itemize}
    \item there is some $r \in \mathbb{R}_{>0}$ such that $A_1 \subseteq r\cdot\mathbb{Z}_{>0}$ and $\sum_{x \in r^{-1}\cdot A_1}x \ll_C n^2$;
    \item we have $|A_2|=O_C(1)$.
\end{itemize}
\end{theorem}

This characterization is tight up to constants.  Indeed, for any disjoint $A_1 \subseteq r\cdot \mathbb{Z}_{>0}$ and $A_2 \subseteq \mathbb{R}_{>0}$, we have
$$|\FS(A_1 \cup A_2)|=|\FS(A_1)+\FS(A_2)| \leq |\FS(A_1)| \cdot |\FS(A_2)| \leq \left(1+\sum_{x \in r^{-1} \cdot A_1}x\right) 2^{|A_2|};$$
this is at most $Cn^2$ whenever (for instance) $\sum_{x \in r^{-1} \cdot A_1}x \leq C^{1/2}n^2-1$ and $|A_2| \leq \log_2(C^{1/2})$.

In fact, we are able to obtain the optimal control on the size of $A_2$.  The implicit constant in the ``bulk'' bound $\sum_{x \in r^{-1} \cdot A_1}x \ll_C n^2$ is necessarily quite large due to our use of the Freiman--Ruzsa Theorem, but we can pin down the optimal value of the implicit constant in the ``junk'' set bound $|A_2|=O_C(1)$: Our \cref{thm:refined} below identifies, for each nonnegative integer $\ell$, the precise threshold for $|\FS(A)|$ which guarantees that one can take $|A_2| \leq \ell$ in the conclusion of \cref{thm: main inverse thm}.  This implies that in the setup of \cref{thm: main inverse thm} (with $n$ sufficiently large in terms of $C$), one can take $|A_2| \leq \lfloor 2C-1 \rfloor$, and this bound is optimal for all values of $C \geq 1/2$.

We also mention that \cref{thm: main inverse thm} can be extended to the setting $A \subseteq \mathbb{R}$, as follows.  Replacing an element of $A$ with its negative merely translates $\FS(A)$ and in particular leaves $|\FS(A)|$ unchanged.  Thus
$$A':=\{x \in \mathbb{R}_{>0}: \{x,-x\} \cap A \neq \emptyset\}$$
is a set of size at least $(n-1)/2$ with $|\FS(A')| \leq |\FS(A)| \ll C|A'|^2$, so \cref{thm: main inverse thm} applies.  Standard arguments with Freiman homomorphisms then allow one to extend the theorem to all torsion-free abelian groups.

More complicated behavior arises if one allows the size of $\FS(A)$ to grow even slightly faster than quadratically.  For example, if $A_1$ is a homogeneous arithmetic progression of length $n-m$ and $A_2$ is any $m$-element set disjoint from $A_1$, then
$$|\FS(A_1 \cup A_2)| \ll n^2 \cdot |\FS(A_2)|;$$
with $m$ fairly small relative to $n$, there is a great variety of choices of $A_2$ (e.g., subsets of higher-rank GAP's) that make $|\FS(A_1 \cup A_2)|$ only slightly superquadratic in $n$.  Once one allows $\FS(A)$ to grow cubically, there are also genuinely multidimensional examples coming from GAPs.

This discussion brings us to our last result, which concerns subsets of $\mathbb{R}^d$ with $d \geq 1$.  The following theorem, which we believe is of independent interest, figures in the proof of \cref{thm: main inverse thm} and was initially motivated by our work on the number of permutation-twisted dot products~\cite{CDK26}. 

\begin{theorem}\label{thm:stability}
For every positive integer $d$ and every real $0<\varepsilon<1$, there is a constant $\gamma=\gamma_d(\varepsilon) > 0$ such that the following holds.  Let $n$ be a positive integer, and let $A \subseteq \mathbb{R}^d$ be an $n$-element set.  If each proper subspace of $\mathbb{R}^d$ contains at most $(1-\varepsilon)n$ elements of $A$, then $|\FS(A)| \geq \gamma n^{d+1}$.
\end{theorem}
A more ``stability-flavored'' phrasing of this theorem is that if $A \subseteq \mathbb{R}^d$ is an $n$-element set with $|\FS(A)|=o(n^{d+1})$, then some codimension-$1$ subspace contains $(1-o(1))n$ elements of $A$.  There are several interesting families of examples saturating this statement.  For instance, if $n_1, \ldots, n_d$ are positive integers and $A \subseteq [0,n_1] \times\cdots\times [0,n_d]$ is any $n$-element set, then
$$\FS(A) \subseteq [0,nn_1] \times \cdots\times [0,nn_d]$$
has size at most $(1+nn_1) \cdots(1+nn_d) \ll n^d n_1 \cdots n_d$; the latter quantity is $\ll n^{d+1}$ if the parameters $n_i$ are chosen to have $n_1\cdots n_d \asymp n$.  This construction provides a rich family of $n$-element subsets $A \subseteq \mathbb{R}^d$ that satisfy $|\FS(A)| \ll n^{d+1}$ and avoid clustering in proper subspaces.  In a related direction, let $n_1, \ldots, n_{d-1}$ be positive integers with $n \leq n_1 \cdots n_{d-1} \ll n$, and let $m=m(n) \in [1,n]$ be a small auxiliary parameter.  Choose any subsets $A_1, A_2 \subseteq [0,n_1] \times \cdots \times [0,n_{d-1}]$ with $|A_1|=n-m$ and $|A_2|=m$, and consider the $n$-element set
$$A:=(A_1 \times \{0\}) \cup (A_2 \times \{1\}) \subseteq \mathbb{R}^d.$$
Then notice that
$$\FS(A) \subseteq [0,nn_1] \times \cdots \times [0,nn_{d-1}] \times [0,m]$$
has size $\ll mn^d$; the latter quantity is $o(n^{d+1})$ if $m$ is chosen to be $o(n)$.  Thus for \emph{any} function $m=o(n)$, no matter how slowly $m/n$ decays, if our $n$-element set $A \subseteq \mathbb{R}^d$ is allowed to have all but $m(n)$ elements in a proper subspace, then there are (many) examples with $|\FS(A)|=o(n^{d+1})$.

An interesting variant of this last problem, which we will discuss further in \Cref{sec:open} and which remains open, is determining the minimum possible size of $\FS(A)$ if $A \subseteq \mathbb{R}^d$ is an $n$-element set such that every $d$-element subset of $A$ spans $\mathbb{R}^d$.

\subsection{Notation and terminology}
We use standard asymptotic notation: We write $f=O(g)$ or $f \ll g$ if $|f| \leq Cg$ for some absolute constant $C>0$.  We write $f \asymp g$ or $f=\Theta(g)$ if $f \ll g$ and $g \ll f$, and we write $f(n)=o(g(n))$ if $f(n)/g(n) \to 0$ as $n \to \infty$.  A subscript $P$ (e.g., $O_P$) indicates that the implicit constant may depend (only) on the parameters $P$.  For integers $m \leq n$, we write $[m,n]:=\{m,m+1, \ldots, n\}$; when $m=1$, we also adopt the shorthand $[n]:=[1,n]$. In $\mathbb{R}^d$, we use the term \emph{subspace} to refer to a linear subspace, and we use the term \emph{affine subspace} to refer to a translate of a linear subspace. 

\subsection{Outline} 
We prove \cref{thm:linear}, \cref{thm:stability}, and \cref{thm: main inverse thm} in \cref{sec:linear}, \cref{sec:Rd}, and \cref{sec:main}, respectively.  In \Cref{sec:open} we comment on related work (including the no-$k$-in-a-line problem) and raise several open questions.

\section{Stability near the minimum}\label{sec:linear}
In this section we prove \Cref{thm:linear}, our exact characterization of subsets of $\mathbb{R}_{>0}$ with close to the minimum possible number of subset sums.

Our argument is driven by the observation that $|\FS(B \cup \{x\}) \setminus\FS(B)|$ is the minimum number of arithmetic progressions of common difference $x$ into which $\FS(B)$ can be decomposed.  In particular, if $B \subseteq \mathbb{R}_{>0}$ and $x$ is strictly larger than all of the elements of $B$, then $$|\FS(B \cup \{x\}) \setminus\FS(B)| \geq |B|+1$$ since the $|B|+1$ elements of $\{0\} \cup B \subseteq \FS(B)$ must all belong to different arithmetic progressions.  The following lemma characterizes when equality can occur.

\begin{lemma}\label{lem:one-step-equality}
Let $m \geq 3$ be an integer.  Let $B \subseteq \mathbb{R}_{>0}$ be an $m$-element set, and let $x \in \mathbb{R}$ be a real number strictly larger than all of the elements of $B$.  Then $$|\FS(B \cup \{x\}) \setminus\FS(B)|=m+1$$
if and only if $$B=\left\{\frac{x}{m+1}, \frac{2x}{m+1}, \ldots, \frac{mx}{m+1}\right\}.$$
\end{lemma}

\begin{proof}
The ``if'' direction is trivial, so we focus on the ``only if'' direction.  Let $B=\{b_1, \ldots, b_m\}$ with $0<b_1<\cdots<b_m<x$, and assume that $$|\FS(B \cup \{x\}) \setminus\FS(B)|=m+1.$$
It follows from the above discussion that $\FS(B)$ can be decomposed into a union of $m+1$ arithmetic progressions of common difference $x$ and that the starting points of these arithmetic progressions are $0,b_1, \ldots, b_m$.  For notational convenience, set $b_{m+1}:=x$.

We show by downward induction that
$$b_1+b_i=b_{i+1}$$
for all $2 \leq i \leq m$.  For the base case $i=m$, we record the inequalities $$b_{m}<b_1+b_{m}<b_1+b_{m+1}.$$
The arithmetic progressions of common difference $b_{m+1}$ starting at the elements of $B$ do not intersect the interval $(b_m,b_1+b_{m+1})$, so the only possibility is $b_1+b_{m}=b_{m+1}$.  For the induction step, we likewise record the inequalities
$$b_i<b_1+b_i<b_1+b_{i+1}=b_{i+2},$$
and the same argument forces $b_1+b_i=b_{i+1}$.

Thus, we can write
$$(b_1, \ldots, b_{m+1})=(b_1,b_2,b_1+b_2,2b_1+b_2, \ldots, (m-1)b_1+b_2);$$
it remains only to show that $b_2=2b_1$.  It is here that we use the condition $m \geq 3$.  We have
$$b_{m+1}=b_1+b_{m}<b_2+b_{m}<b_2+b_{m+1}$$
(note that $b_2+b_m \in \FS(B)$ since $b_2,b_m$ are distinct), which implies that
$$b_2+b_m=b_1+b_{m+1}.$$
Unraveling the definitions gives
$$(m-2)b_1+2b_2=mb_1+b_2,$$
which rearranges to $b_2=2b_1$.
\end{proof}

We are now ready to prove \Cref{thm:linear}.

\begin{proof}[Proof of \Cref{thm:linear}]
The ``if'' direction of the theorem is easy.  Suppose that $a>0$ and $A \subseteq a \cdot \mathbb{Z}_{>0}$ is an $n$-element set with
$$\sum_{x \in a^{-1} \cdot A}x \leq \binom{n+1}{2}+M.$$
Then each element of $\FS(A)$ is of the form $ma$ for some integer $0 \leq m \leq \binom{n+1}{2}+M$, whence $|\FS(A)| \leq \binom{n+1}{2}+1+M$.

We now turn to the ``only if'' direction.  Let $0 \leq M \leq n-4$, and assume that $A \subseteq \mathbb{R}_{>0}$ is an $n$-element set with $|\FS(A)| \leq \binom{n+1}{2}+1+M$.  Write $A=\{a_1, \ldots, a_n\}$, where $0<a_1<\cdots<a_n$.  For each $1 \leq i \leq n$, define the set $$A(i):=\{a_1, \ldots, a_i\}$$
and the quantity
$$y_i:=|\FS(A(i)) \setminus \FS(A(i-1))|-i$$
(with the convention $A(0):=\emptyset$ and $\FS(\emptyset)=\{0\}$).  From
$$1+\sum_{i=1}^n |\FS(A(i)) \setminus \FS(A(i-1))|=|\FS(A)|\leq \binom{n+1}{2}+1+M$$
we obtain the key inequality
$$\sum_{i=1}^n y_i \leq M.$$

By the observation mentioned at the beginning of this section, $y_1,\ldots,y_n$ are all nonnegative integers, so $y_i>0$ for at most $M$ indices $i$.  In particular, we have $y_{i_0}=0$ for some $i_0 \geq n-M$.  Since $i_0-1 \geq n-M-1 \geq 3$, \Cref{lem:one-step-equality} with $B:=A(i_0-1)$ and $x:=a_{i_0}$ tells us that
$$(a_1, \ldots, a_{i_0})=(a_1,2a_1, \ldots, i_0 a_1).$$
It follows that $$\FS(A(i_0))=a_1 \cdot\left\{0,1, \ldots, \binom{i_0+1}{2}\right\}.$$
Let $i_1$ be the largest index such that $\FS(A(i_1))$ is an arithmetic progression of difference $a_1$.  We have just seen that $i_1 \geq i_0$, and we will complete the proof by showing that in fact $i_1=n$.

Assume for the sake of contradiction that $i_1<n$.  Consider
$$\FS(A(i_1+1))=\FS(A(i_1)) \cup (a_{i_1+1}+\FS(A(i_1))).$$
Since $\FS(A(i_1+1))$ is not an arithmetic progression of common difference $a_1$, the translate $a_{i_1+1}+\FS(A(i_1))$ must be disjoint from $\FS(A(i_1))$.  Then
$$y_{i_1+1}=|\FS(A(i_1))|-(i_1+1)\geq \binom{i_1+1}{2}-i_1.$$
Notice that $A(i_1+1)$ is not a homogeneous arithmetic progression.  Thus, for each $i_1+2 \leq  i \leq n$, the set $A(i)$ also fails to be a homogeneous arithmetic progression, and \Cref{lem:one-step-equality} gives $y_i \geq 1$.  Now the estimate
\begin{equation}\label{eq:stability-y's-ineq}
\sum_{i=1}^n y_i \geq y_{i_1+1}+\sum_{i=i_1+2}^n y_i \geq \binom{i_1+1}{2}-i_1+(n-i_1-1)=n+\frac{i_1(i_1-3)}{2}-1 \geq n-1>M
\end{equation}
gives a contradiction.

We have shown that $\FS(A)=\FS(A(n))$ is an arithmetic progression with common difference $a_1$.  It follows that $A \subseteq a_1 \cdot \mathbb{Z}_{>0}$ and
$$\FS(A)=a_1 \cdot \left\{0, 1, \ldots, \sum_{x \in a_1^{-1} \cdot A}x \right\}.$$
Thus, we have
$$\sum_{x \in a_1^{-1} \cdot A}x=|\FS(A)|-1 \leq \binom{n+1}{2}+M,$$
as desired.
\end{proof}

The preceding argument breaks down for $M=n-3$ essentially because \Cref{lem:one-step-equality} fails for $m=2$ (equality there holds whenever $B=\{b_1,b_2\}$ and $x=b_1+b_2$).  In order to get a contradiction in \eqref{eq:stability-y's-ineq} with $M=n-3$, one would need $i_1 \geq 2$, but in the absence of \Cref{lem:one-step-equality}, the best that one can guarantee is $i_1 \geq 1$.

\section{Sumsets in $\mathbb{R}^d$}\label{sec:Rd} 

In this section we prove \cref{thm:stability}, our result about subset sums in $\mathbb{R}^d$.

We induct on the dimension $d$ and the size $n$ of our set $A \subseteq \mathbb{R}^d$. The main idea is to show that we can always find distinct elements $a,b \in A$ such that $\FS(A)$ is significantly larger than $\FS(A \setminus \{a,b\})$.  More precisely, writing $A':= A \setminus \{a,b\}$, we always have the inclusion 
$$\FS(A)=\FS(A')+\{0,a,b,a+b\} \supseteq \FS(A')+\{a,b\}=b+\FS(A' \cup \{a-b\}).$$
The size of the right-hand side is $|\FS(A')|$ plus the minimum number of arithmetic progressions of common difference $a-b$ into which $\FS(A')$ can be decomposed; a lower bound for the latter is the size of the projection of $\FS(A')$ onto the orthogonal complement of $a-b$ in $\mathbb{R}^d$.\footnote{A more natural strategy would be to remove points one at a time rather than in pairs, but complications arise if $A$ has many points on a small number of lines through the origin.}

We need a bit of notation.  For each nonzero $v \in \mathbb{R}^d$, let $v^\perp$ denote the orthogonal complement of $v$, and let $\pi_v: \mathbb{R}^d \to v^\perp$ be the orthogonal projection.  The observation from the previous paragraph can be expressed as
\begin{equation}\label{eq:proj-ineq}
|\FS(A)| \geq \max_{a,b \in A \text{ distinct}} \left(|\FS(A \setminus \{a,b\})|+|\pi_{a-b}(\FS(A \setminus \{a,b\}))| \right).
\end{equation}
We would like to identify a choice of $a,b \in A$ for which $\pi_{a-b}(\FS(A \setminus \{a,b\}))$ is large.  To this end, for each nonzero vector $v \in \mathbb{R}^d$ and each finite set $B \subseteq \mathbb{R}^d$, let $E_B(v)$ denote the collection of $2$-element subsets $\{b_1,b_2\} \subseteq B$ such that $\pi_v(b_1)=\pi_v(b_2)$; thus, $E_B(v)$ records the ``collisions'' that result from projecting $B$ onto $v^\perp$.  The following lemma shows that for any subset $A \subseteq \mathbb{R}^d$, there are distinct $a,b \in A$ such that $E_A(a-b)$ is small.

\begin{lemma} \label{lem: find v}
Let $d \geq 2$ and $n \geq 3$ be integers, and let $A \subseteq \mathbb{R}^d$ be an $n$-element set that is not contained in any line.  Then there are distinct $a, b\in A$ such that $|E_A(a-b)| \leq n/2$.
\end{lemma}

\begin{proof}
Let $\pi: \mathbb{R}^d \to \mathbb{R}^2$ be a projection that is injective on $A$ and has the property that $\pi(A)$ is not contained in any line (a generic projection will work).  It is a classical result of Ungar~\cite{U82} that every $n$-element non-collinear point set in the plane determines lines in at least $n-1$ different directions.  In our setting, Ungar's theorem produces pairs $(a_1,b_1), \ldots, (a_{n-1},b_{n-1}) \in A \times A$ such that the vectors $\pi(a_1)-\pi(b_1), \ldots, \pi(a_{n-1})-\pi(b_{n-1})$ are all non-parallel (and in particular nonzero).  It follows that the vectors $a_1-b_1, \ldots, a_{n-1}-b_{n-1}$ are all non-parallel. This ensures that each $2$-element subset of $A$ is contained in $E_A(a_i-b_i)$ for at most a single index $i$.  Thus
$$\sum_{i=1}^{n-1}|E_A(a_i-b_i)| \leq \binom{n}{2},$$
and by averaging there is some $i$ such that $|E_A(a_i-b_i)| \leq n/2$.
\end{proof}

We wish to apply \eqref{eq:proj-ineq} with the elements $a,b \in A$ produced by \cref{lem: find v}.  First, however, we must check that the set $A \setminus \{a,b\}$ remains ``sufficiently $d$-dimensional'' for us to apply our induction hypothesis.  The following notion of ``sufficiently $d$-dimensional'', which is finer than the notion in the statement of \cref{thm:stability}, will do the trick.  For positive integers $d,m,n$ with $m \geq d-1$, let $\Xi_d(n,m)$ denote the collection of $n$-element subsets $A \subseteq \mathbb{R}^d$ such that each proper subspace of $\mathbb{R}^d$ contains at most $m$ elements of $A$; the latter condition is equivalent to requiring that every $(m+1)$-element subset of $A$ spans all of $\mathbb{R}^d$.  Our quantity of interest is
\[
f_d(n,m):=\min\{|\FS(A)|:A\in\Xi_{d}(n,m)\},
\]
which is weakly increasing in $n$ and weakly decreasing in $m$. We are interested in lower-bounding $f_d(n, (1-\varepsilon)n)$. \cref{lem: find v} and \eqref{eq:proj-ineq} lead to the following recursive lower bound for $f_d(n,m)$.

\begin{lemma} \label{lem:refined}
Let $d,n,m$ be integers with $n-2\geq m\geq d-1 \geq 1$ and $m\geq 3$. Then
\[
f_{d}(n, m)\geq f_d(n-2, m) + \min_{\frac{n-4}{2}\leq\ell\leq n-2}f_{d-1}\left(\ell, \ell-\left\lceil \frac{(n-2-m)^2}{2n-2-m}\right\rceil\right).
\]
\end{lemma}

\begin{proof}
Choose $A \in \Xi_d(n,m)$ with $|\FS(A)|=f_d(n,m)$.  Notice that $A$ is not contained in any line through the origin because it spans $\mathbb{R}^d$. Because $|\FS(A)|$ does not change if we negate some elements of $A$, we may assume that $A$ is moreover not contained in any affine line.  Thus, \cref{lem: find v} provides distinct elements $a,b \in A$ such that $|E_A(a-b)| \leq n/2$.  For brevity, set $v:=a-b$ and $A':=A \setminus \{a,b\}$.  It is clear that $A' \in \Xi_d(n-2,m)$, so by definition $|\FS(A')| \geq f_d(n-2,m)$. 

We now turn to
$\FS(\pi_{v}(A'))$. First, notice that trivially
$$\ell:=|\pi_{v}(A')| \geq (n-2)-|E_A(v)| \geq \left\lceil\frac{n-4}{2}\right\rceil.$$
Next, we claim that 
$$\pi_{v}(A') \in \Xi_{d-1}\left(\ell, \ell-\left\lceil \frac{(n-m-2)^2}{2n-2-m}\right\rceil\right),$$
where we have identified $v^\perp$ with $\mathbb{R}^{d-1}$.  The claim amounts to showing that no proper subspace of $v^\perp$ contains too many elements of $\pi_{v}(A')$.  Let $H \subseteq v^{\perp}$ be a proper subspace.  Since $A'\in \Xi_{d}(n-2,m)$, there is an $(n-2-m)$-element set $$W\subseteq A' \setminus \pi_{v}^{-1}(H).$$
Let $w_1,\ldots, w_k$ be the elements of $\pi_v(W)$; of course $w_1,\ldots,w_k\not\in H$. Convexity gives
$$\frac{n}{2} \geq |E_A(v)| \geq |E_{W}(v)|=\sum_{i=1}^k \binom{|\pi_{v}^{-1}(w_i)\cap W|}{2} \geq k \binom{|W|/k}{2}= \frac{|W|^2}{2k} - \frac{|W|}{2},$$
where the penultimate inequality used that $\sum_i |\pi_{v}^{-1}(w_i)\cap W|=|W|$.  This rearranges to
\[
k\geq \frac{|W|^2}{n+|W|} = \frac{(n-2-m)^2}{2n-2-m},
\]
which establishes the claim. The conclusion follows. 
\end{proof}

We prove \cref{thm:stability} by iteratively applying \cref{lem:refined}. 

\begin{proof}[Proof of~\cref{thm:stability}] 
We induct on the dimension $d$.  The base case $d=1$ follows from the result of Nathanson~\cite{N95} mentioned in the introduction: Any set of $n$ real numbers contains at least $(n-1)/2$ nonzero elements of the same sign and hence has $\Omega(n^2)$ distinct subset sums.

For the induction step, let $d \geq 2$ and $\varepsilon > 0$, and assume that $n$ is sufficiently large in terms of $d, \varepsilon$.  Applying \cref{lem:refined} $\left\lfloor\varepsilon n/4\right\rfloor$ times (note each application decreases $n$ by $2$) yields
\begin{align*}
f_{d}\left(n, \lceil (1-\varepsilon) n\rceil\right) &\geq 1+ \sum_{i=0}^{\left\lfloor\varepsilon n/4\right\rfloor-1} \min_{\frac{n-2i-4}{2}\leq\ell\leq n-2i-2}f_{d-1}\left(\ell, \ell-\left\lceil\frac{((n-2i)-2-\lceil (1-\varepsilon) n\rceil)^2}{2(n-2i)-2-\lceil (1-\varepsilon) n\rceil}\right\rceil\right).
\end{align*}
There is some $\varepsilon'=\varepsilon'(\varepsilon) >0$ such that
$$\left\lceil\frac{(n-2i-2-\lceil (1-\varepsilon) n\rceil)^2}{2(n-2i)-2-\lceil (1-\varepsilon) n\rceil}\right\rceil \geq \varepsilon' \ell$$
holds for all pairs $(i,\ell)$ under consideration in this sum.
Thus, due to the monotonicity of $f_d$, the induction hypothesis with $d-1$, $\varepsilon'$ gives
\begin{align*}
    f_d(n, \lceil (1-\varepsilon)n\rceil) &\geq 1 + \sum_{i=0}^{\lceil \varepsilon n/4\rceil-1} \min_{\frac{n-2i-4}{2}\leq\ell\leq n-2i-2} f_{d-1}(\ell, \lceil (1-\varepsilon')\ell\rceil) \\
    &\geq \lceil \varepsilon n/4\rceil \cdot \gamma_{d-1}(\varepsilon') \left(\frac{(n-2\lceil \varepsilon n/4\rceil-4)}{2}\right)^{d} \\
    &\gg_{d, \varepsilon} n^{d+1},
\end{align*}
as required. 
\end{proof}

\section{Stability farther from the minimum}\label{sec:main}

We now turn to \cref{thm: main inverse thm}, our characterization of $n$-element sets of positive reals with $O(n^2)$ subset sums.  Let us sketch the main steps of the argument.  Let $C>0$ be a constant, and suppose that $A$ is a set of $n$ positive reals with $|\FS(A)| \leq Cn^2$.  

We begin by pigeonholing to find a subset $A' \subseteq A$ of size $|A'| \geq n/2$ (say) with doubling $O_C(1)$.  The Freiman--Ruzsa Theorem produces a proper symmetric GAP $Q$ of rank $r=O_C(1)$ in which $A'$ is $\Omega_C(1)$-dense.  Using the generators of $Q$ as coordinates, we identify $Q$ with a box $\widetilde{Q}$ of size $|Q|$ in $\mathbb{Z}^r$ that is centered at the origin.  Let $\pi: \mathbb{Z}^r \to \mathbb{R}$ denote the projection map that sends $\widetilde{Q}$ to $Q$, and identify $A'$ with its lift $\widetilde{A'} =\pi^{-1}(A')\subseteq \widetilde{Q}\subset \mathbb{R}^r$. Since $\FS(A') = \pi(\FS(\widetilde{A'}))$, we will use the subset sums of $\widetilde{A'}$ as a proxy for the subset sums of $A'$. 

A technical ``cleaning'' step (see \cref{lem:cleaning-2} below) lets us reduce to the case where no proper subspace of $\mathbb{R}^r$ contains more than half (say) of the elements of $\widetilde{A'}$. Now \cref{thm:stability} guarantees the expansion
$$|\FS(\widetilde{A'})| \gg_C |\widetilde{A'}|^{r+1} \gg_C n^{r+1}.$$
The constraint $|\FS(A')| \leq |\FS(A)| \leq C n^2$ forces $\pi$ to be ``very non-injective'' on $\FS(\widetilde{A'})$. In particular, the enlarged box $n\widetilde{Q}$, which certainly contains $\FS(\widetilde{A'})$, satisfies
\begin{equation}\label{eq:fiber bound}
    \max_{u\in \mathbb{R}} |\pi^{-1}(u) \cap n\widetilde{Q}| \geq \frac{|\FS(\widetilde{A'})|}{|\FS(A')|} \gg_C n^{r-1}.
\end{equation}
Drawing on ideas from the geometry of numbers, we show (see \cref{prop:proj'} below) that the left-hand side of \eqref{eq:fiber bound} can be this large only if all of the generators of $Q$ are small integer multiples of a single real number; this ``collapse to rank $1$'' is the most involved part of the argument.  We deduce that $A'=\pi(\widetilde{A'})$ is contained in a homogeneous arithmetic progression of length $O_C(n)$.

The last step is upgrading this structural characterization of $A'$ to a structural characterization of $A$.  To achieve this, we use the fact that $\FS(A')$ is a dense ($\Omega_C(1)$-fraction) subset of an interval of length $\Theta_C(n^2)$ and hence $\FS(A)=\FS(A \setminus A')+\FS(A')$ cannot contain too many disjoint translates of $\FS(A')$. The remainder of this section fleshes out the details.

\subsection{Finding a piece of local structure}\label{subsec:small-doubling}

We begin by locating a large additively structured subset of $A$; we will only ever use the following lemma with $\delta=1/2$ (any constant would do).

\begin{lemma}\label{lem:smalldoubling}
Let $C>0$ and $0<\delta<1$ be reals, and let $n$ be a positive integer.  If $A \subseteq \mathbb{R}_{>0}$ is an $n$-element set with $|\FS(A)| \leq Cn^2$, then there is a subset $A' \subseteq A$ of size $|A'| \geq \delta n$ such that 
$$|A'+A'| \ll_{C,\delta}|A'|.$$
\end{lemma}

\begin{proof}
We may assume that $n \geq \delta^{-1}$ (say), as otherwise the lemma is trivial.  As in the proof of \cref{thm:linear}, write $A=\{a_1, \ldots, a_n\}$, where $0<a_1<\cdots<a_n$, and define $A(i):=\{a_1, \ldots, a_i\}$.  Further define
$$z_i := |\FS(A(i+1)) \setminus \FS(A(i))|,$$
so that
$$\sum_{i=0}^{n-1} z_i =|\FS(A)|-1 \leq Cn^2.$$
It follows that there is some index $i\geq \delta n \geq 1$ such that $z_i \leq 2(1-\delta)^{-1}Cn=O_{C,\delta}(n)$ (say).

We now show that $A(i)$ has doubling $O_{C,\delta}(1)$.  Recall that $z_i$ is the minimum number of arithmetic progressions of common difference $a_{i+1}$ into which $\FS(A(i))$ can be decomposed.  Since each such arithmetic progression contains at most $2$ elements of the interval $[0,2a_i]$, we have
$$2z_i \geq |\FS(A(i)) \cap [0,2a_i]| \geq |A(i) \widehat{+}A(i)| \geq |A(i)+A(i)|-i,$$
where $A(i) \widehat{+}A(i)=\{x+y:x,y\in A(i),\, x\neq y\}$ is the restricted sumset.  Rearranging gives
$$|A(i)+A(i)| \leq 2z_i+i \ll_{C,\delta} i=|A(i)|,$$
and we conclude the lemma with $A':=A(i)$.
\end{proof}

The set $A'$ produced by this lemma is ripe for an application of the Freiman--Ruzsa Theorem (see, e.g., \cite[Theorem 5.33]{TV}), which says that if $C>0$ is a constant and $B \subseteq \mathbb{R}$ is a finite set with $|B+B| \leq C|B|$, then $B$ is contained in a proper symmetric GAP of rank $O_C(1)$ and size $O_C(|B|)$.  Recall that a \emph{symmetric GAP} is a set of the form
\begin{equation}\label{eq:GAP}
    Q:=\left\{\sum_{j=1}^{r} t_j x_j : -s_j \leq x_j\leq s_j,\;
  x_j \in \mathbb{Z} \right\}=t_1 \cdot [-s_1,s_1]+\cdots +t_r \cdot [-s_r,s_r],
\end{equation} 
where $t_1, \ldots, t_r$ are nonzero reals and $s_1, \ldots, s_r$ are positive integers.  The \emph{rank} of $Q$ is $r$, and we say that $Q$ is \emph{proper} if it has size $|Q|=(2s_1+1) \cdots (2s_r+1)$ (i.e., the various choices of $(x_1, \ldots, x_r)$ all produce distinct elements of $Q$).\footnote{The Freiman--Ruzsa Theorem is often stated without the symmetry and properness conditions that we included here.  It is easy to see that any GAP $Q$ of rank $r$ is contained in a symmetric GAP $Q'$ of rank at most $r+1$ and size $|Q'| \ll_r |Q|$.  A general fact about GAPs (see \cite[Theorem 3.40]{TV}) then implies that $Q'$ is contained in a proper symmetric GAP $Q''$ of the same rank and size $|Q''| \ll_r |Q'| \ll_r |Q|$.}  
We say that $t_1, \ldots, t_r$ are the \emph{common differences} of $Q$ and that $2s_1+1, \ldots, 2s_r+1$ are its \emph{side lengths}. Every GAP $Q$ comes with a natural projection map $\pi=\pi_{t_1, \ldots, t_r}\colon \ZZ^r \to \RR$ given by $$\pi(x_1, \ldots, x_r):= t_1 x_1+\cdots+ t_r x_r.$$
By definition, $Q$ is the image under $\pi$ of the discrete box $$\widetilde{Q}:=[-s_1,s_1] \times \cdots \times [-s_r,s_r];$$
this map is bijective if and only if $Q$ is proper, in which case we can identify each subset $R \subseteq Q$ with its \emph{lift} $\widetilde{R}:=\pi^{-1}(R) \cap \widetilde{Q}$. 

In the setting of \cref{thm: main inverse thm}, \cref{lem:smalldoubling} with $\delta:=1/2$ provides some $A' \subseteq A$ of size $|A'| \gg_C n$ such that $|A'+A'| \ll_C |A'|$, and the Freiman--Ruzsa Theorem produces a proper symmetric GAP $Q\supseteq A'$ of rank $r=O_C(1)$ and size $O_C(n)$. In the next several steps of our argument, we will study the subset sums of $\widetilde{A'}$ (as a proxy for the subset sums of $A'$).

\subsection{Cleaning step} \label{subsec:cleaning} We wish to use \cref{thm:stability} to show that the set $\widetilde{A'}\subseteq \mathbb{R}^r$ produced in the previous subsection must have many different subset sums. This theorem does not apply, however, if $\widetilde{A'}$ happens to be mostly contained in a proper subspace of $\RR^r$, so first we perform a ``cleaning'' step that identifies the dimension $1 \leq r' \leq r$ in which a large subset $\widetilde{A''}\subseteq \widetilde{A'}$ is robustly supported.  The following lemma encapsulates our cleaning procedure.

\begin{lemma}\label{lem:cleaning-2} 
Let $r$ be a positive integer, and let $0 < \delta,\varepsilon < 1$ be reals.  Let $Q$ be a proper symmetric GAP of rank $r$, and let $B \subseteq Q$ be a subset of size $|B| \geq \delta |Q|$. Then there is a subset $B' \subseteq B$ of size $|B'| \gg_{r,\varepsilon} |B|$ that is contained in a proper symmetric GAP $Q'$ of rank $r' \leq r$ and size $|Q'| \ll_{r,\delta,\varepsilon} |Q|$ with the property that each proper subspace of $\RR^{r'}$ contains at most $(1-\varepsilon)|B'|+1$ elements of $\widetilde{B'}$.
\end{lemma}

We will prove \cref{lem:cleaning-2} by iterating the following ``slicing'' lemma.

\begin{lemma}\label{lem:cleaning-technical}
Let $r$ be a positive integer, and let $0<\eta<1$ be a real.  Let $Q$ be a proper symmetric GAP of rank $r$, with associated lift $\widetilde{Q} \subseteq \mathbb{Z}^r$ and projection map $\pi$.  If $V \subseteq \mathbb{R}^r$ is a codimension-1 subspace with $|\widetilde{Q} \cap V| \geq \eta |\widetilde{Q}|$, then there is a  proper symmetric GAP $Q'$ of rank at most $r-1$ and size $|Q'| \ll_{r,\eta} |Q|$ that contains $\pi(\widetilde Q\cap V)$. 
\end{lemma}

\begin{proof}
Let us write $$\widetilde{Q} = \{(x_1,\dots,x_r)\in\mathbb{Z}^r : -s_j\leq x_j \leq s_j\} \quad \text{and} \quad \pi(x_1, \ldots, x_r)=t_1x_1+\cdots+t_rx_r,$$ 
where $s_1, \ldots, s_r$ are positive integers and $t_1, \ldots, t_r$ are nonzero reals.  Also write $$V = \{v_1 x_1 + \cdots + v_r x_r = 0\},$$
where $(v_1, \ldots, v_r)\in\mathbb{R}^r\setminus\{0\}$.

Consider the set of indices $J := \{j\in[r] : v_j \neq 0\}$. Let us extend $\pi$ linearly to all of $\mathbb R^r$ and write $\pi\vert_{\mathbb Z^J}$ for the restriction of $\pi$ to $\mathbb Z^J$. For each $j \in J$, the hyperplane $V$ is the graph of a function on the coordinates $[r] \setminus\{j\}$, so
$$|\widetilde{Q} \cap V| \leq \frac{|\widetilde{Q}|}{2s_j+1} \leq \frac{\eta^{-1}|\widetilde{Q}\cap V|}{2s_j+1}$$
and $s_j \ll_\eta 1$.  Write $\widetilde{Q} = \widetilde{Q}_J \times \widetilde{Q}_{-J}$, where $\widetilde{Q}_J$ and $\widetilde{Q}_{-J}$ correspond to the coordinates of $J$ and of $[r] \setminus J$, respectively.  Since $V$ constrains only the $J$-coordinates, we have $V \cap \widetilde{Q} = W \times \widetilde{Q}_{-J}$, where
$$W:= \left\{(x_j)_{j\in J}\in \mathbb{Z}^J : \sum_{j\in J} v_j x_j = 0, \, -s_j\leq x_j\leq s_j\right\}.$$
Moreover, since $s_j = O_\eta(1)$ for all $j\in J$ and there are only $O_r(1)$ possibilities for $J$, there are only $O_{r,\eta}(1)$ possibilities for $W$.

Consider a single such $W$.  By construction, $W$ is contained in a sublattice of $\ZZ^J$ of rank $r_W \leq |J|-1$.  Fix a $\ZZ$-basis $\{b_1, \ldots, b_{r_W}\}$ for this sublattice.  Then there is a positive integer $s_W=O_{r,\eta}(1)$ such that
$$W \subseteq \left\{\sum_{j=1}^{r_W} b_j x_j
  : ~-s_W \leq x_j \leq s_W, ~x_j \in \ZZ\right\}.$$
It follows that $\pi(\widetilde Q\cap V)=\pi(W \times \widetilde{Q}_{-J})$ is contained in the symmetric GAP
$$Q':=\left\{\sum_{j=1}^{r_W} \pi\vert_{\mathbb Z^J}(b_j) x_j +\sum_{i \in [r] \setminus J} t_i y_i  : ~-s_W \leq x_j \leq s_W, ~-s_i \leq y_i \leq s_i, ~x_j,y_i \in \ZZ \right\}$$
of rank $r_W+(r-|J|) \leq r-1$ and size $$|Q'| \leq (2s_W+1)^{r_W} \cdot \prod_{j \notin J}(2s_j+1) \ll_{r,\eta} |Q|.$$ \cite[Theorem 3.40]{TV} lets us make $Q'$ proper at the cost of increasing its size by at most $O_r(1)$.
\end{proof}

Our elementary proof of \cref{lem:cleaning-technical} was based on the observation that the hypothesis $|\widetilde{Q} \cap V| \geq \eta|\widetilde{Q}|$
severely constrains the geometry of $\widetilde{Q}\cap V$.  One could alternatively use the discrete John theorem from lattice reduction theory (see \cite{TV08,vHK}) and the fact that the intersection of $V$ with the convex hull of $\widetilde{Q}$ is convex.  Bringing in this additional (nontrivial) tool would let us drop the assumption $|\widetilde{Q} \cap V| \geq \eta|\widetilde{Q}|$.

\begin{proof}[Proof of \cref{lem:cleaning-2}]
We induct on the rank $r$.  The base case $r=1$ is trivial.  For the induction step, let $Q,B$ be as in the lemma statement for some $r \geq 2$.  If each codimension-1 subspace of $\RR^{r}$ contains at most $(1-\varepsilon)|B|+1$ elements of $\widetilde{B}$, then we obtain the desired conclusion with $B':=B$ and $Q':=Q$.  Suppose instead that there is a codimension-1 subspace $V \subseteq \mathbb{R}^r$ that contains at least $(1-\varepsilon)|B|+1$ elements of $\widetilde{B}$.  In particular, $V$ contains at least $(1-\varepsilon)\delta|Q|$ elements of $\widetilde{Q}$.  \cref{lem:cleaning-technical} with $\eta:=(1-\varepsilon)\delta$ provides a proper symmetric GAP $Q_1$ of rank $r' \leq r-1$ and size $|Q_1| \ll_{r,\delta, \varepsilon}|Q|$ that contains $\pi(\widetilde Q\cap V)$.  \emph{A fortiori} $\pi(\widetilde Q\cap V)$ contains $B_1 := \pi(\widetilde B \cap V)$, which has size 
$$|B_1| \geq (1-\varepsilon)|B| \geq (1-\varepsilon)\delta |Q| \gg_{r, \delta, \varepsilon} |\widetilde{Q}_1|.$$  
Applying the induction hypothesis to $B_1 \subseteq Q_1$ produces the desired subset $B' \subseteq B_1 \subseteq B$  and proper symmetric GAP $Q'\supseteq B'$, which have sizes 
$$|B'| \gg_{r,\varepsilon} |B_1| \gg_{r,\varepsilon} |B| \quad \text{and} \quad
|Q'| \ll_{r,\delta,\varepsilon} |Q_1| \ll_{r,\delta,\varepsilon} |Q|$$
and also satisfy the final property of the lemma statement.
\end{proof}

We apply \cref{lem:cleaning-2} to the set $A'$ from the previous subsection with $\varepsilon:=0.51$ and a suitable choice of $\delta=\delta(C)>0$.  This provides some further subset $A'' \subseteq A'$ of size $|A''| \gg_C n$ that is contained in another proper symmetric GAP $Q'$, of rank $r'=O_C(1)$ and size $O_C(n)$, with the additional property that each proper subspace of $\mathbb{R}^{r'}$ contains at most half of the elements of $\widetilde{A''}$.

\subsection{Reducing to rank $1$}\label{subsec:reducing}

So far we have worked with proper GAPs.  As outlined in the proof overview, our next task is to study the conditions under which the projection map of a non-proper GAP may have very large fibers.  Such a result will be useful for leveraging the fact that the cleaned set $A''$ output by \cref{lem:cleaning-2} has few subset sums  (even though the projection map $\pi$ is bijective on $\widetilde{Q'}$, it is not necessarily bijective on the enlarged box in which $\FS(\widetilde{A''})$ is naturally contained).  The analysis of this subsection is phrased in terms of projection maps on boxes in $\ZZ^r$.

\begin{proposition}\label{prop:proj'}
Fix a positive integer $r$ and constants $\varepsilon,\gamma,C>0$. Let $n$ be a sufficiently large positive integer.  Let $I_1,\dots,I_r\subseteq \mathbb{Z}$ be
finite intervals such that $|I_j|\geq\varepsilon n$ for all $1\leq j\leq r$ and 
\[
\gamma n^{r+1}\le |I_1|\cdots |I_r|\le \varepsilon^{-1}n^{r+1}.
\]
Let $t_1,\dots,t_r$ be nonzero reals, and define the projection map $\pi:\mathbb Z^r\to\mathbb{R}$ by
\[
\pi(x_1,\dots,x_r):=t_1x_1+\cdots+t_rx_r.
\]
Suppose $X\subseteq I_1\times\cdots\times I_r$ satisfies $|X|\ge \gamma n^{r+1}$ and
$|\pi(X)|\le Cn^2$. Then there is a positive real number $a$ such that
$t_1,\dots,t_r\in a\cdot \mathbb{Z}$ and
\[
|a^{-1}t_j|\cdot|I_j|\ll_{r,\varepsilon,\gamma,C} n^2
\quad \text{for all }1\le j\le r.
\]
\end{proposition}

Since the proof of \cref{prop:proj'} is somewhat technical, we first describe the simplified argument in the special case of rank $r=2$.  The assumptions $|I_1 \times I_2| \geq |X| \geq \gamma n^3$ and $|\pi(X)| \leq Cn^2$ together imply that $$\gamma C^{-1}n \leq \max_{u \in \RR}|\pi^{-1}(u)\cap (I_1\times I_2)|=\max_{u \in \RR} |\{(x_1,x_2) \in I_1 \times I_2: \, t_1x_1+t_2x_2=u\}|.$$
The right-hand side is at most $1$ if $t_1/t_2$ is irrational, so we must have $(t_1,t_2)=a(\lambda_1,\lambda_2)$ for some coprime nonzero integers $\lambda_1,\lambda_2$.  Fix some $u \in \mathbb{R}$ achieving the maximum.  All pairs $(x_1,x_2)$ with $t_1x_1+t_2x_2=u$ differ from one another by integer multiples of $(\lambda_2, -\lambda_1)$.  In particular, since the $x_1$'s all lie in the interval $I_1$, the number of pairs $(x_1,x_2)$ is at most $1+|I_1|/|\lambda_2|$; likewise, the number of such pairs is at most $1+|I_2|/|\lambda_1|$.  It follows that
$$n \ll_{\gamma,C} \gamma C^{-1}n-1 \leq \min\left\{\frac{|I_1|}{|\lambda_2|}, \frac{|I_2|}{|\lambda_1|}\right\}=\frac{|I_1| \cdot |I_2|}{\max\{|\lambda_1| \cdot |I_1|,|\lambda_2| \cdot |I_2|\}} \leq \frac{\varepsilon^{-1}n^3}{\max\{|\lambda_1| \cdot |I_1|,|\lambda_2| \cdot |I_2|\}}.$$
Rearranging and recalling that $(t_1,t_2)=a(\lambda_1,\lambda_2)$ gives the desired inequality.

When the rank is larger than $2$, it becomes more complicated both to deduce that $(t_1, \ldots, t_r)$ is a scalar multiple of an integer vector $(\lambda_1, \ldots, \lambda_r)$ and subsequently to obtain the $\max_j \{|\lambda_j| \cdot |I_j|\}$ savings in the fiber bound.  The following lemma achieves both goals.


\begin{lemma}\label{lem:slice} 
Let $r$ be a positive integer.  Let
$\lambda=(\lambda_1,\dots,\lambda_r)\in (\mathbb{R}\setminus\{0\})^r$,
and let $I_1,\dots,I_r\subset \mathbb{Z}$ be nonempty finite intervals.  Define $\pi: \mathbb Z^r \to \RR$ by $\pi(x_1, \ldots, x_r):=\lambda_1 x_1+\cdots+\lambda_r x_r$.   Let $s \in \mathbb{R}$ be any real number.  Then the following holds.
\begin{enumerate}
    \item If $\pi^{-1}(s)\cap(I_1\times\cdots\times I_r)$ has affine dimension at most $r-2$, then $$|\pi^{-1}(s)\cap(I_1\times\cdots\times I_r)| \leq \frac{|I_1| \cdots |I_r|}{\max_j |I_j| \cdot \min_j |I_j|}.$$
    \item If $\lambda_1, \ldots, \lambda_r$ are integers with $\gcd(\lambda_1, \ldots, \lambda_r)=1$ and $\pi^{-1}(s)\cap(I_1\times\cdots\times I_r)$ has affine dimension $r-1$, then $$|\pi^{-1}(s)\cap(I_1\times\cdots\times I_r)| \ll_r \frac{|I_1| \cdots |I_r|}{\max_j (|\lambda_j| \cdot |I_j|)}.$$
\end{enumerate}
\end{lemma}

\begin{proof}
We begin with part (1).  Note that $r \geq 2$.  Let $j_1 \in [r]$ be an index with $|I_{j_1}|=\max_j |I_j|$.  The assumption $\lambda_{j_1} \neq 0$ guarantees (e.g., by Gaussian elimination) that there is some $j_2 \in [r] \setminus \{j_1\}$ such that each point of $\pi^{-1}(s)$ is determined by its coordinates on the indices in $[r] \setminus \{j_1,j_2\}$.  Thus
$$|\pi^{-1}(s)\cap(I_1\times\cdots\times I_r)| \leq \prod_{j \in [r] \setminus\{j_1,j_2\}} |I_j|=\frac{|I_1| \cdots |I_r|}{|I_{j_1}| \cdot |I_{j_2}|} \leq \frac{|I_1| \cdots |I_r|}{\max_j |I_j| \cdot \min_j |I_j|}.$$

We now turn to part (2).  Note that $s \in \mathbb{Z}$ since otherwise $\pi^{-1}(s)\cap(I_1\times\cdots\times I_r)$ would be empty.  By possibly translating the intervals $I_j$ (recall that $\gcd(\lambda_1, \ldots, \lambda_r)=1$), we can reduce to the case where $s=0$ and the origin is contained in the box $I_1 \times \cdots \times I_r$.  Consider the $(r-1)$-dimensional subspace $V \subseteq \RR^r$ given by
$$V:=\{(x_1, \ldots, x_r) \in \RR^r: \lambda_1x_1+\cdots+\lambda_rx_r=0\}.$$
The lattice
$$\Lambda:=V \cap \ZZ^r \subseteq V$$
has rank $r-1$ and covolume $\|\lambda\|_2$ (due to the primitivity of $\lambda$).
  
We also need the continuous box
$$D:=(I_1 \times \cdots \times I_r)+(-1/2,1/2)^r \subseteq \RR^r,$$
which has side lengths $|I_1|, \ldots, |I_r|$, and the slice
$$D_0:=D \cap V,$$
which is a bounded convex set of dimension $r-1$.  Notice that $\pi^{-1}(0)\cap(I_1\times\cdots\times I_r)=\Lambda \cap D_0$.  It is a standard fact from the geometry of numbers (see, e.g., \cite[Lemma 3.26]{TV}) that\footnote{For the reader's convenience, we also include the following elementary proof, which we obtained from ChatGPT~5.4 Pro before learning about the reference \cite[Lemma 3.26]{TV}: Let $K$ denote the convex hull of $\pi^{-1}(0)$. Then $K$ is a polytope with all of its vertices in $\Lambda$. The last assumption of the lemma guarantees that $K$ has dimension $r-1$, so we can decompose $K$ into at least $|\pi^{-1}(0)\cap(I_1\times\cdots\times I_r)|-(r-1)$ simplices with vertices in $\Lambda$ and with pairwise disjoint interiors.  Each simplex necessarily has volume at least $\covol_{r-1}(\Lambda)/(r-1)!=\|\lambda\|_2/(r-1)!$.
Since $K$ is contained in $D_0$, a comparison of volumes gives
\begin{equation*}
|\pi^{-1}(0)\cap(I_1\times\cdots\times I_r)| \leq \frac{(r-1)! \cdot \vol_{r-1}(D_0)}{\|\lambda\|_2}+(r-1) \ll_r \frac{\vol_{r-1}(D_0)}{\|\lambda\|_2}+1,
\end{equation*} which for our later application is just as good as \eqref{eq:vol-comparison}.}
\begin{equation}\label{eq:vol-comparison}
|\pi^{-1}(0)\cap(I_1\times\cdots\times I_r)|\ll_r \frac{\vol_{r-1}(D_0)}{\covol(\Lambda)}=\frac{\vol_{r-1}(D_0)}{\|\lambda\|_2}.
\end{equation}

We next fix an index $j_1 \in [r]$ with $|\lambda_{j_1}| \cdot |I_{j_1}|=\max_j (|\lambda_{j}| \cdot |I_{j}|)$ and consider the projection $\pi': \RR^r \to \RR^{r-1}$ that removes the $j_1$-th coordinate.  Parameterizing the points of $V$ by
$$x_{j_1}=-\lambda_{j_1}^{-1}\sum_{j\neq j_1} \lambda_j x_j$$
and computing the corresponding Jacobian, we find that the restriction of $\pi'$ to $V$ dilates volumes by a factor of $|\lambda_{j_1}|/\|\lambda\|_2$.  Since $D_0 \subseteq D$, we have
\begin{align*}
\vol_{r-1}(D_0)&=\frac{\|\lambda\|_2}{|\lambda_{j_1}|} \cdot \vol_{r-1}(\pi'(D_0))\\
 &\leq \frac{\|\lambda\|_2}{|\lambda_{j_1}|} \cdot \vol_{r-1}(\pi'(D))\\
 &=\frac{\|\lambda\|_2}{|\lambda_{j_1}|} \cdot \prod_{j \neq j_1} |I_j|\\
 &=\|\lambda\|_2 \cdot \frac{|I_1| \cdots |I_r|}{|\lambda_{j_1}| \cdot |I_{j_1}|}\\
 &=\|\lambda\|_2 \cdot \frac{|I_1| \cdots |I_r|}{\max_j (|\lambda_{j}| \cdot |I_{j}|)}.
\end{align*}
Plugging this bound into \eqref{eq:vol-comparison} gives the desired conclusion.
\end{proof}

\begin{proof}[Proof of \cref{prop:proj'}]
The case $r=1$ is immediate, since for $a:=|t_1|$ we have
\[
|a^{-1}t_1|\,|I_1|=|I_1|\le \varepsilon^{-1}n^2.
\]
We henceforth restrict our attention to the case $r\ge 2$.  We may assume that $n$ is sufficiently large relative to $\varepsilon,\gamma,C$.  To set the stage for applying \cref{lem:slice}, we observe that
\begin{equation}\label{eq:prop43-fiber-lower}
\max_{u \in \RR}|\pi^{-1}(u)\cap (I_1\times\cdots\times I_r)|\ge\max_{u \in \RR}|\pi^{-1}(u)\cap X|\ge \frac{|X|}{|\pi(X)|}\ge \gamma C^{-1}n^{r-1};
\end{equation}
fix some $u$ achieving this maximum.  Notice also that
$$\max_j |I_j| \geq (|I_1| \cdots |I_r|)^{1/r} \geq \gamma^{1/r} n^{1+1/r}.$$
If $\pi^{-1}(u)\cap(I_1\times\cdots\times I_r)$ has affine dimension at most $r-2$, then part (1) of \cref{lem:slice} (with $(t_1, \ldots, t_r)$ playing the role of $(\lambda_1, \ldots, \lambda_r)$ and with $s:=u$) gives
$$|\pi^{-1}(u)\cap(I_1\times\cdots\times I_r)| \leq \frac{|I_1| \cdots |I_r|}{\max_j |I_j| \cdot \min_j |I_j|} \leq \frac{\varepsilon^{-1} n^{r+1}}{\gamma^{1/r}n^{1+1/r} \cdot \varepsilon n}=\gamma^{-1/r}\varepsilon^{-2}n^{r-1-1/r}.$$
This contradicts \eqref{eq:prop43-fiber-lower} if $n$ is sufficiently large relative to $\varepsilon,\gamma,C$. Hence we may assume that the set $\pi^{-1}(u)\cap(I_1\times\cdots\times I_r)$ has affine dimension $r-1$. It follows that $(t_1, \ldots, t_r)$ is a scalar multiple of an integer vector.  Write
$$(t_1, \ldots, t_r)=a(\lambda_1, \ldots, \lambda_r),$$
where $a>0$ is a real number and $\lambda_1, \ldots, \lambda_r$ are nonzero integers satisfying $\gcd(\lambda_1, \ldots, \lambda_r)=1$.  Part (2) of \cref{lem:slice} (with $s:=a^{-1}u$) now tells us that
$$|\pi^{-1}(u)\cap(I_1\times\cdots\times I_r)| \ll_r \frac{|I_1| \cdots |I_r|}{\max_j (|\lambda_j| \cdot |I_j|)}\leq \frac{\varepsilon^{-1}n^{r+1}}{\max_j (|\lambda_j| \cdot |I_j|)}.$$
Combining this with \eqref{eq:prop43-fiber-lower} gives
$$\max_j (|a^{-1}t_j| \cdot |I_j|)=\max_j (|\lambda_j| \cdot |I_j|) \ll_r \varepsilon^{-1}\gamma^{-1}C n^2 \ll_{r,\varepsilon,\gamma,C} n^2,$$
as desired. 
\end{proof}

We now have the necessary tools to show that our set $A''$ is contained in a homogeneous arithmetic progression of length $O_C(n)$; in other words, we can ``reduce'' the rank of the GAP $Q'$ from $r'$ to $1$.

\begin{lemma}\label{lem:piece-of-structure}
For every $C>0$, there is some $C'=C'(C)>0$ such that the following holds.  Let $n$ be a positive integer, and suppose that $A \subseteq \mathbb{R}_{>0}$ is an $n$-element set with $|\FS(A)| \leq Cn^2$.  Then there are a real $a \in \mathbb{R}_{>0}$ and a subset $A' \subseteq A$ of size $|A'| \gg_C n$ such that $A' \subseteq a \cdot [1,\lfloor C'n\rfloor]$.
\end{lemma}

\begin{proof} 
We apply \cref{lem:smalldoubling}, the Freiman--Ruzsa Theorem, and \cref{lem:cleaning-2} as described at the ends of \cref{subsec:small-doubling} and \cref{subsec:cleaning}.  This produces a subset $A'\subseteq A$ of size $|A'| \gg_C n$ and a proper symmetric GAP $Q\supseteq A'$ of rank $r = O_C(1)$ and size $|Q|\ll_C n$, with the property that no proper subspace of $\mathbb{R}^r$ contains more than half of the elements of $\widetilde{A'}$. Then \cref{thm:stability} with $\varepsilon=1/2$ yields
$$|\FS(\widetilde{A'})| \gg_C |\widetilde{A'}|^{r+1} \gg_C  n^{r+1}.$$
Let us parameterize the GAP $Q$ as in \eqref{eq:GAP}. Since $\widetilde{A'} \subseteq [-s_1,s_1] \times \cdots \times [-s_r,s_r]$ and $|\widetilde{A'}| \leq n$, we have
$$\FS(\widetilde{A'}) \subseteq [-ns_1,ns_1] \times \cdots \times [-ns_r,ns_r].$$
The properness of $Q$ ensures that
$$n\ll_C |A'| \leq (2s_1+1) \cdots (2s_r+1)=|Q| \ll_C n,$$
so we can apply \cref{prop:proj'} with the intervals $I_j:=[-ns_j,ns_j]$ for $1 \leq j \leq r$. For suitable values of $\gamma, \varepsilon>0$ (depending only on $C$), this proposition produces a constant $C'=C'(C)>0$ and some $a \in \mathbb{R}_{>0}$ such that $t_1, \ldots, t_r \in a \cdot \mathbb{Z}$ and $$|a^{-1}t_j| \cdot (2ns_j+1) \leq r^{-1}C' n^2$$
for each $1 \leq j \leq r$.  It follows that each $|a^{-1}t_j|s_j \leq r^{-1}C' n$ and hence (omitting floor functions)
$$A' \subseteq Q= a \cdot \sum_{j=1}^r a^{-1}t_j \cdot [-s_j,s_j] \subseteq a \cdot \sum_{j=1}^r [-r^{-1}C' n , r^{-1}C' n] \subseteq a \cdot [-C'n,C'n].$$
Since $A' \subseteq A \subseteq \mathbb{R}_{>0}$ by assumption, we in fact conclude that $A' \subseteq a \cdot [1,C'n]$, as desired.
\end{proof}

\subsection{From local to global structure}

\cref{lem:piece-of-structure} identified a large structured sub-piece of any set $A$ with few subset sums.  The following combinatorial lemma lets us upgrade this local structure to a global constraint on the entire set $A$.

\begin{lemma}\label{lem:local-to-global}
For every choice of $\varepsilon, C>0$, there is some positive integer $D:=D(\varepsilon,C)>0$ such that the following holds.  Let $n$ be a positive integer.  Suppose $S \subseteq [0,\lfloor\varepsilon^{-1}n^2\rfloor]$ is a set of integers of size $|S| \geq \varepsilon n^2$.  If $R$ is a set of positive reals of size at most $n$ such that $|\FS(R)+S| \leq Cn^2$, then there is a decomposition $R=R_1 \sqcup R_2$ such that:
\begin{itemize}
    \item $R_1 \subseteq D^{-1} \cdot \mathbb{Z}_{>0}$ and $\sum_{x \in R_1} x \ll_{\varepsilon,C} n^2$;
    \item  we have $|R_2| \ll_{\varepsilon,C} 1$.
\end{itemize}
\end{lemma}
\begin{proof}
Recall that $\FS(R)+S$ is the union of the translates $y+S$ over $y \in \FS(R)$. Since $|S| \geq \varepsilon n^2$ and $|\FS(R)+S| \leq Cn^2$, there can be at most $m:=\lceil C\varepsilon^{-1}\rceil$ pairwise disjoint such translates; we will make repeated use of this observation.

Define the threshold $T := \varepsilon^{-1} n^2+1$, and let $R_{\mathrm{big}}$ denote the set of elements of $R$ of size at least $T$. We claim that $|R_{\mathrm{big}}|< m$: Indeed, if $R_{\mathrm{big}}$ contained distinct elements $x_1, \ldots, x_m$, then the $m+1$ translates
$$S,\quad x_1+S,\quad (x_1+x_2)+S,\quad \dots,\quad (x_1+\cdots+x_m)+S$$
in $\FS(R)+S$ would be pairwise disjoint, in contradiction with our earlier observation. 

Next, we show that the remaining elements have sum $O_{\varepsilon,C}(n^2)$. Write
$R\setminus R_{\mathrm{big}}=\{y_1,\dots,y_\ell\}$, and consider the partial sums $$p_i:= y_1 + \cdots + y_i$$ for $0 \leq i \leq \ell$.  Greedily extract a subsequence of indices $0=i_0<i_1<\cdots <i_r \leq \ell$ such that $p_{i_{j}}-p_{i_{j-1}} \geq T$ for each $1 \leq j \leq r$.  Since the $r+1$ translates
$$S, \quad p_{i_1}+S, \quad p_{i_2}+S, \quad \ldots, \quad p_{i_r}+S$$
in $\FS(R)+S$ are pairwise disjoint, we have $r \leq m-1$.  Also, since every $y_i \leq T$, we have $p_{i_{j}}-p_{i_{j-1}} \leq 2T$ for each $1 \leq j \leq r$, and we have $p_{i_r} \geq y_1+\cdots+y_\ell-T$.  Thus
$$y_1+\cdots+y_\ell \leq T+p_{i_r}=T+\sum_{j=1}^r (p_{i_j}-p_{i_{j-1}}) \leq (2r+1)T \leq (2m-1)T \ll_{\varepsilon,C} n^2.$$

It remains to show that all but $O_{\varepsilon,C}(1)$ elements of $R\setminus R_{\mathrm{big}}$ lie in $D^{-1}\cdot \mathbb{Z}_{>0}$ for some positive integer $D=O_{\varepsilon,C}(1)$. For a nonzero real number $x$, define its \emph{residue class modulo $1$} as its image in $\mathbb{R}/\mathbb{Z}$, and its \emph{(additive) order modulo $1$} as the smallest positive integer $N$ such that $Nx\in\mathbb{Z}$ (if $x$ is irrational, then its additive order modulo $1$ is $\infty$). Since $S\subseteq \mathbb{Z}$, if $x,y\in R$ have distinct residue classes modulo $1$, then the translates $x+S$ and $y+S$ are disjoint. Hence, by the initial observation, the elements of $R$ occupy at most $m$ residue classes modulo $1$.

Let $R_{\mathrm{bad}}$ be the set of elements in $R\setminus R_{\mathrm{big}}$ with order greater than $m$ modulo $1$. We claim that each residue class modulo $1$ contains at most $m$ elements of $R_{\mathrm{bad}}$: Indeed, if some residue class contained at least $m$ such elements, then the corresponding partial sums would produce at least $m+1$ distinct residue classes modulo $1$, which is a contradiction. We conclude that $|R_{\mathrm{bad}}|\leq m^2=O_{\varepsilon,C}(1)$.

Finally, set $R_1: = R\setminus (R_{\mathrm{big}}\cup R_{\mathrm{bad}})$ and $R_2 :=R_{\mathrm{big}}\cup R_{\mathrm{bad}}$ and $D:= \lcm(1, 2, \dots, m) \ll_{\varepsilon,C} 1$. Since each element of $R_1$ has order at most $m$ modulo $1$, it is immediate that $R_1 \subseteq D^{-1} \cdot \ZZ_{>0}$.  Thus the decomposition $R = R_1\sqcup R_2$ satisfies the conclusion of the lemma.
\end{proof}

At last, we deduce \cref{thm: main inverse thm}.

\begin{proof}[Proof of \cref{thm: main inverse thm}]
\cref{lem:piece-of-structure} produces a positive constant $C'=O_C(1)$, some positive real $a$, and a subset $A' \subseteq A$ of size $|A'|\gg_C n$ such that $A' \subseteq a \cdot [1,\lfloor C'n \rfloor]$.  The first result of Nathanson mentioned in the introduction (see also the $M=0$ case of \cref{thm:linear}) gives
$$|\FS(A')| \geq \binom{|A'|+1}{2}+1 \gg |A'|^2 \gg_C n^2.$$
Now apply \cref{lem:local-to-global} with $R:= a^{-1} \cdot (A \setminus A')\subseteq \mathbb{R}_{>0}$ and $S:=\FS(a^{-1} \cdot A')\subseteq \mathbb{Z}$ and a suitable choice of $\varepsilon>0$ (depending only on $C$).  This lemma produces a positive integer $D \ll_{C} 1$ and a decomposition $A \setminus A'=A'_1 \sqcup A_2$, where $A'_1 \subseteq D^{-1}a\cdot \mathbb{Z}_{>0}$ and
$$\sum_{x \in a^{-1} \cdot A'_1}x \ll_C n^2 \quad \text{and} \quad |A_2|\ll_C 1.$$ 
We now obtain the conclusion of the theorem with $r:=D^{-1} a$ and $A_1:=A' \cup A'_1 \subseteq r \cdot \mathbb{Z}_{>0}$, since
\begin{equation*}
\sum_{x \in r^{-1} \cdot A_1}x=\sum_{x \in Da^{-1} \cdot A'_1}x+\sum_{x \in Da^{-1} \cdot A'}x \leq D\cdot O_{C}(n^2)+n \cdot DC'n \ll_C n^2. \qedhere
\end{equation*}
\end{proof}

\subsection{Controlling the size of the junk set}
In this final subsection, we establish the sharp bounds on $|A_2|$ in \cref{thm: main inverse thm} that we mentioned in the introduction.  The precise statement is as follows.

\begin{theorem}\label{thm:refined}
Let $\ell$ be a nonnegative integer, and let $n$ be a positive integer that is sufficiently large relative to $\ell$. If $A \subseteq \mathbb{R}_{>0}$ is an $n$-element set with
\begin{equation}\label{eq:junk-tight-constraint}
|\FS(A)| <(\ell+2)\left(\binom{n-\ell}{2}+1\right)+\binom{\ell+2}{3},
\end{equation}
then there is a decomposition $A=A_1 \sqcup A_2$, where: 
\begin{itemize}
    \item there is some $r \in \mathbb{R}_{>0}$ such that $A_1 \subseteq r\cdot\mathbb{Z}_{>0}$ and $\sum_{x \in r^{-1}\cdot A_1}x \ll_\ell n^2$;
    \item we have $|A_2|\leq \ell$.
\end{itemize}
\end{theorem}

It follows from \cref{prop:calculation-for-tightness} below that the condition~\eqref{eq:junk-tight-constraint} is tight.  In particular, in \cref{thm: main inverse thm}, one can take $|A_2| \leq \lfloor 2C-1 \rfloor$ (for $n$ sufficiently large in terms of $C$), and this bound is optimal. 

To prove \cref{thm:refined}, we will start with the decomposition of $A$ provided by \cref{thm: main inverse thm} and then derive further constraints in the style of the proof of \cref{lem:local-to-global}.  The central calculation, from which the expression in \eqref{eq:junk-tight-constraint} arises, is encapsulated in the following proposition.

\begin{proposition}\label{prop:calculation-for-tightness}
Let $\ell$ be a nonnegative integer, and let $n \geq \ell+1$ be a positive integer.  Let $B_0 \subseteq \mathbb{R}_{>0}$ be an $(n-\ell-1)$-element set, and let $B_1 \subseteq \mathbb{R}$ be an $(\ell+1)$-element set.  Set $B:=(B_0 \times \{0\}) \cup (B_1 \times \{1\}) \subseteq \mathbb{R}^2$.  Then
$$|\FS(B)| \geq (\ell+2)\left(\binom{n-\ell}{2}+1\right)+\binom{\ell+2}{3}.$$
Equality is attained if $B_0$ is a homogeneous arithmetic progression and $B_1$ is an arithmetic progression, with the same common difference; moreover, when $n-5 \geq \ell \geq 3$, this is the only way for equality to be attained.
\end{proposition}


\begin{proof}
For $0 \leq j \leq \ell+1$, define the $j$-fold restricted sumset of a set $A$ by
$$j^{\wedge}A:=\{a_1+\cdots+a_j: ~a_1, \ldots, a_j \in A \text{ distinct}\}.$$
With this notation, we can express
$$\FS(B)=\bigsqcup_{j=0}^{\ell+1} \left( \FS(B_0)+j^{\wedge}B_1  \right) \times \{j\}.$$
It is easy to check (see~\cite[Theorems 1 and 2]{N95}) that for each $0 \leq j\leq \ell+1$, the quantity $|j^{\wedge}B_1|$ is minimized when $B_1$ is an arithmetic progression, in which case
$$|j^{\wedge}B_1|=j(\ell+1)-j^2+1;$$
moreover, arithmetic progressions are the unique minimizers when $2 \leq j \leq \ell-1$.  At the same time, Nathanson's theorem (mentioned in the introduction) and \cref{thm:linear} tell us that
$$|\FS(B_0)| \geq \binom{n-\ell}{2}+1,$$
and that when $n-\ell-1 \geq 4$, the unique minimizers are homogeneous arithmetic progressions.  It follows from the preceding two centered equations that
\begin{equation}\label{eq:layer-sumset}
|\FS(B_0)+j^{\wedge}B_1| \geq |\FS(B_0)|+|j^{\wedge}B_1|-1 \geq \binom{n-\ell}{2}+j(\ell+1)-j^2+1
\end{equation}
for each $0 \leq j \leq \ell+1$.  Summing gives
\begin{align}\label{eq:layers-combined}
|\FS(B)| &\geq \sum_{j=0}^{\ell+1} \left(\binom{n-\ell}{2}+j(\ell+1)-j^2+1  \right)\\
 &=(\ell+2)\left(\binom{n-\ell}{2}+1\right)+\binom{\ell+2}{3},\notag 
\end{align}
which establishes the first part of the proposition.

Suppose that $n-5 \geq \ell \geq 3$ and equality holds in \eqref{eq:layers-combined}.  Then both inequalities in \eqref{eq:layer-sumset} are equalities for all $j$.  Let us examine the $j=2$ instance.  Equality in the second inequality of \eqref{eq:layer-sumset} implies that $B_0$ is a homogeneous arithmetic progression with some common difference $r_0>0$ and $B_1$ is an arithmetic progression with some common difference $r_1>0$.  Equality in the first inequality of \eqref{eq:layer-sumset} implies that $r_0=r_1$.  It is easy to check that equality holds in \eqref{eq:layers-combined} whenever $B_0$ is a homogeneous arithmetic progression and $B_1$ is an arithmetic progression, with the same common difference.
\end{proof}

This proposition lets us discuss a key example that will guide the proof of \cref{thm:refined}.  Let $B_0:=[n-\ell-1]$ and $B_1:=[\ell+1]$.  For $\alpha>0$, consider the set
$$A:=B_0 \cup (\alpha+B_1) \subseteq \mathbb{R},$$
which is the image of $B:=([n-\ell-1] \times \{0\}) \cup ([\ell+1] \times \{1\})$ under the map $\psi: (x,y) \mapsto x+\alpha y$.  Suppose that either $\alpha$ has order at least $\ell+2$ modulo $1$ (notably including the irrational $\alpha$ case) or $\alpha$ is sufficiently large (e.g., at least $10(\ell+1) n^2$).  One should think of this condition as a quantitative sense, either non-archimedean or archimedean, in which $\alpha$ is ``additively independent of small integers''.  Then the map $\psi$ is a sort of Freiman isomorphism on $B$ for restricted sums, in the sense that two subset sums in $B$ coincide if and only if the corresponding subset sums in $A$ coincide.  Thus $|A|=|B|=n$, and \cref{prop:calculation-for-tightness} tells us that $|\FS(A)|=|\FS(B)|$ is the quantity appearing on the right-hand side of~\eqref{eq:junk-tight-constraint}.  If $\alpha$ is transcendental, for instance, then every dilate of $\mathbb{Z}_{>0}$ will miss at least $\ell+1$ elements of $A$; this shows the optimality of the condition~\eqref{eq:junk-tight-constraint}.

We also isolate a simple lemma about subset sums in $\mathbb{R}/\mathbb{Z}$.

\begin{lemma}\label{lem:junk}
Let $\ell$ be a nonnegative integer.  Let $A \subseteq \mathbb{R}$ be an $\ell$-element set of numbers each with order at least $\ell+1$ modulo $1$.  Then $\FS(A)$ meets at least $\ell+1$ distinct residue classes modulo $1$.  If we further assume that all of the elements of $A$ have order at least $\ell+2$ modulo $1$, then equality holds if and only if the absolute values of the elements of $A$ are all equal modulo $1$.
\end{lemma}

\begin{proof}
We proceed by induction on $\ell$.  The base case $\ell=0$ is trivial.  For the induction step, suppose that $\ell \geq 1$.  Pick an element $x \in A$.  If $\FS(A \setminus \{x\})$ already meets at least $\ell+1$ residue classes modulo $1$, then there is nothing to show.  Otherwise, $\FS(A \setminus \{x\})$ meets exactly $\ell$ residue classes modulo $1$ (by induction), and hence its image in $\mathbb{R}/\mathbb{Z}$ is not invariant under adding $x$.  It follows that $\FS(A)$ meets strictly more residue classes modulo $1$, as desired.  The equality condition is easily checked inductively.
\end{proof}

\begin{proof}[Proof of \cref{thm:refined}]
To begin, \Cref{thm: main inverse thm} with $C:=(\ell+3)/2$ (for instance) provides a constant $K=K(\ell)\geq 1$ and a decomposition $A=A'_1 \sqcup A'_2$, where:
\begin{itemize}
    \item there is some $r' \in \mathbb{R}_{>0}$ such that $A'_1 \subseteq r'\cdot\mathbb{Z}_{>0}$ and $\sum_{x \in (r')^{-1}\cdot A'_1}x \leq K n^2$;
    \item we have $|A'_2|\leq K$.
\end{itemize}
Since dilating $A$ by a positive real changes neither the hypothesis nor the conclusion of \Cref{thm:refined}, we may assume that $r'=1$.   To obtain the desired decomposition $A=A_1 \sqcup A_2$, we will add some elements of $A'_2$ to $A'_1$ to obtain our ``bulk'' set $A_1$, and the remaining elements of $A'_2$ will form our ``junk'' set $A_2$.  Notice that for $A_1$ to satisfy the conclusion of \cref{thm:refined}, it can include only elements of $A'_2$ that have both size $O_\ell(n^2)$ and order $O_\ell(1)$ modulo $1$.  Instead of using a fixed threshold for ``small size'' and ``small order modulo $1$'', we will identify the ``okay'' elements of $A'_2$ one at a time according to a greedy procedure.

Let us be more precise.  We will split
$$A'_2=A_{\mathrm{big}} \sqcup A_{\mathrm{bad}} \sqcup A_{\mathrm{okay}},$$ roughly as in the proof of \Cref{lem:local-to-global}.  
Write $A'_2=\{a_1, \ldots, a_k\}$ with $0<a_1<\cdots<a_k$ and $k \leq K$.
With the convention $a_0:=Kn^2$, let $k_0$ denote the smallest index $j \in [k]$ such that $a_{j}>10K a_{j-1}$, if such a $j$ exists; then set $A_{\mathrm{big}}:=\{a_{k_0}, \ldots, a_k\}$.  (The choice of the constant $10$ is not important.)  We now iteratively assign the elements of $A'_2 \setminus A_{\mathrm{big}}$ to $A_{\mathrm{bad}}$ or $A_{\mathrm{okay}}$.  At each stage of our procedure, let $L$ denote the lcm of the orders modulo $1$ of the elements that we have already assigned to $A_{\mathrm{okay}}$ (starting with $L=1$ at the beginning of the procedure).  As long as there is a remaining element $x \in A'_2 \setminus A_{\mathrm{big}}$ such that $Lx$ has order at most $(K+\ell+2)$ modulo $1$, assign $x$ to $A_{\mathrm{okay}}$, and update $L$ accordingly.  Once there are no such elements, assign those that remain to $A_{\mathrm{bad}}$.

We now obtain a new decomposition $A=A_1 \sqcup A_2$ by setting
$$A_1:=A'_1 \cup A_{\mathrm{okay}} \quad \text{and} \quad A_2:=A_{\mathrm{big}} \sqcup A_{\mathrm{bad}}.$$
Our construction of $A_{\mathrm{okay}}$ guarantees that the elements of $A_1$ all have order at most $(K+\ell+2)^K\ll_\ell 1$ modulo $1$ (since $L$ increased by a factor of at most $K+\ell+2$ at each step), and that their sum is at most $$Kn^2(1+10K+(10K)^2+\cdots+(10K)^{k_0-1}) \ll_\ell n^2.$$
Let us further dilate $A$ by the lcm of the orders modulo $1$ of the elements of $A_1$, so that $A_1 \subseteq \mathbb{Z}_{>0}$.  The set $A_1$ still has sum $O_\ell(n^2)$, so it satisfies the conclusion of the theorem.  It remains to show that $|A_2| \leq \ell$.

The rest of the proof proceeds by a careful analysis of the fact that $\FS(A)$ is the union of the translates of $\FS(A_1)$ by the elements of $\FS(A_2)$.  Let us write
$$\ell_{\mathrm{big}}:=|A_{\textrm{big}}| \quad \text{and} \quad \ell_{\mathrm{bad}}:=|A_{\textrm{bad}}|,$$
so that $|A_2|=\ell_{\mathrm{big}}+\ell_{\mathrm{bad}}$.  We will begin by showing that either $\ell_{\mathrm{big}}+\ell_{\mathrm{bad}} \leq \ell$ or $\{\ell_{\mathrm{big}},\ell_{\mathrm{bad}}\}=\{0,\ell+1\}$, and then we will conclude the argument by eliminating the latter.  Two relevant properties of the decompositions constructed above are that all elements of $A_{\textrm{bad}}$ have order at least $K+\ell+3$ modulo $1$, and that all elements of $A_{\textrm{big}}$ have size at least $5M$, where $$M:=\sum_{x \in A_1}x+\sum_{x \in A_{\textrm{bad}}}x.$$  For the latter property, recall that the elements of $A'_1$ sum to at most $Kn^2 \leq a_{k_0-1} \leq Ka_{k_0-1}$; the elements of $A_{\mathrm{okay}} \sqcup A_{\mathrm{bad}}=A'_2 \setminus A_{\mathrm{big}}$ sum to at most 
$$|A'_2 \setminus A_{\mathrm{big}}| \cdot a_{k_0-1} \leq K a_{k_0-1};$$
and all elements of $A_{\textrm{big}}$ have size at least $10Ka_{k_0-1}$.

Since $|A_1|=n-O_\ell(1)$, Nathanson's result gives $|\FS(A_1)|\geq n^2/2-O_\ell(n)$.  It follows that $\FS(A)$ contains at most $\ell+2$ disjoint translates of $\FS(A_1)$ (as long as $n$ is sufficiently large).  \Cref{lem:junk} provides elements $b_1, \ldots, b_{\ell_{\mathrm{bad}}+1} \in \FS(A_{\mathrm{bad}})$ whose residues modulo $1$ are all distinct.  Write $A_{\textrm{big}}=\{c_1, \ldots, c_{\ell_{\textrm{big}}}\}$.  Then the $\ell_{\textrm{big}}+1$ partial sums
$$0, \quad c_1, \quad c_1+c_2, \quad \ldots, \quad c_1+c_2+\cdots+c_{\ell_{\textrm{big}}}$$
in $\FS(A_{\textrm{big}})$ all differ by at least $5M$.  Consider the $(\ell_{\textrm{bad}}+1)(\ell_{\textrm{big}}+1)$ elements of $\FS(A_2)$ of the form
$$b_i+(c_1+\cdots+c_j),$$
with $1 \leq i \leq \ell_{\textrm{bad}}+1$ and $0 \leq j \leq \ell_{\textrm{big}}$.  These elements are all distinct, and moreover their translates of $\FS(A_1)$ are all disjoint.  Indeed, given an element of the form $x=b_i+(c_1+\cdots+c_j)+y$ with $y \in \FS(A_1)$, one can recover $j$ by looking at the approximate size of $x$ (due to the fact that the $c$'s are larger than $M$), and then one can recover $i$ from the residue class of $x-(c_1+\cdots+c_j)$ modulo $1$.  It follows that $\FS(A)$ contains at least $(\ell_{\textrm{bad}}+1)(\ell_{\textrm{big}}+1)$ disjoint translates of $\FS(A_1)$, and the observation from the beginning of this paragraph gives
$$(\ell_{\textrm{bad}}+1)(\ell_{\textrm{big}}+1) \leq \ell+2.$$
Thus either $\ell_{\mathrm{big}}+\ell_{\mathrm{bad}} \leq \ell$ or $\{\ell_{\mathrm{big}},\ell_{\mathrm{bad}}\}=\{0,\ell+1\}$.  The last step of the proof is showing that the latter possibility is impossible. 

Assume for the sake of contradiction that $\{\ell_{\mathrm{big}},\ell_{\mathrm{bad}}\}=\{0,\ell+1\}$.  We will construct an auxiliary set $B \subseteq \mathbb{R}^2$ with $|A|=|B|$ and $|\FS(A)|=|\FS(B)|$ which satisfies the hypotheses of \cref{prop:calculation-for-tightness}; this proposition will then produce the desired contradiction.  One should think of $B$ as a model set which is Freiman-isomorphic to $A$ with respect to subset sums.  This strategy requires establishing several structural properties of $A_2$---recall the key example described after \cref{prop:calculation-for-tightness}---which look slightly different according to which of $A_{\mathrm{big}},A_{\mathrm{bad}}$ is nonempty.

We begin with the case where $\ell_{\mathrm{big}}=\ell+1$ and $\ell_{\mathrm{bad}}=0$.  We will show that $A_{\mathrm{big}}$ is contained in a (real) translate of an interval of $O_\ell(1)$ consecutive integers.  This is vacuously true if $\ell=0$.  For $\ell \geq 1$, we observe that all of the elements of $A_{\mathrm{big}}$ have the same residue modulo $1$: Indeed, if (for instance) $c_1,c_2$ had distinct residues modulo $1$, then the $\ell+3$ translates of $\FS(A_1)$ by
$$0, \quad c_1, \quad c_2, \quad c_1+c_2, \quad c_1+c_2+c_3, \quad \ldots, \quad c_1+c_2+\cdots+c_{\ell+1}$$
would all be disjoint, which is impossible.  Let us now express the elements of $A_{\mathrm{big}}$ as $$X, X+d_1, \ldots, X+d_{\ell},$$
where $X \geq 5M$ is a real and $0<d_1<\cdots<d_{\ell}$ are positive integers.

We claim that $d_\ell=O_\ell(1)$.  Recall that $|\FS(A_1)| \geq \binom{n-\ell}{2}+1$ (by Nathanson's result), and consider the $\ell+2$ disjoint translates of $\FS(A_1)$ by the partial sums of the sequence
$$X+d_\ell, \quad X, \quad X+d_1, \quad X+d_2, \quad \ldots, \quad X+d_{\ell-1}.$$
These provide $(\ell+2)\left(\binom{n-\ell}{2}+1  \right)$ distinct elements of $\FS(A)$.  Additionally, the translate of $\FS(A_1)$ by $X$ provides
$$|(X+\FS(A_1)) \setminus (X+d_\ell+\FS(A_1))|=|(d_\ell+\FS(A_1)) \setminus \FS(A_1)|$$
further elements of $\FS(A)$.  The latter quantity is the minimum number of arithmetic progressions of common difference $d_\ell$ into which $\FS(A_1)$ can be decomposed.  Since $\FS(A_1) \subseteq [0,Kn^2]$, we see that each arithmetic progression of common difference $d_\ell$ contains at most $Kn^2/d_\ell+1$ elements of $\FS(A_1)$, so
$$|(d_\ell+\FS(A_1)) \setminus \FS(A_1)| \geq \frac{\binom{n-\ell}{2}+1}{Kn^2/d_\ell+1}.$$
Our assumption on $|\FS(A)|$ guarantees that the latter quantity is smaller than $\binom{\ell+2}{3}$, and rearranging gives $d_\ell=O_\ell(1)$ provided $n\gg_\ell 1$. 

We are now ready to apply \cref{prop:calculation-for-tightness}.  Consider the auxiliary set
$$B:=(A_1 \times \{0\}) \cup (\{0,d_1, \ldots, d_\ell\} \times \{1\}) \subseteq \mathbb{Z}^2,$$
which is obtained from $A$ by sending each element $i+jX$ (with $j \in \{0,1\}$ and $0 \leq i \leq Kn^2$) to $(i,j)$.  Observe that for each element $x \in \FS(A)$, one can uniquely determine how many summands came from $A_1$ and how many came from $A_{\mathrm{big}}$: The number of summands from $A_{\mathrm{big}}$ is precisely $\lfloor x/X \rfloor$ (since the elements of $A_1$ are all small and the elements of $A_{\mathrm{big}}$ are all slightly larger than $X$).  It follows that two subset sums of $A$ coincide if and only if the corresponding subset sums in $B$ coincide, so $|\FS(A)|=|\FS(B)|$.  Now \cref{prop:calculation-for-tightness} provides the desired contradiction.

It remains to treat the case where $\ell_{\mathrm{bad}}=\ell+1$ and $\ell_{\mathrm{big}}=0$.  We proceed along much the same lines as in the previous case.  Since $\FS(A)$ contains at most $(\ell+2)$ disjoint translates of $\FS(A_1)$, \cref{lem:junk} provides some $\alpha \in (0,1)$, of order at least $\ell+2$ modulo $1$, such that $A_{\mathrm{bad}} \subseteq \{-\alpha,\alpha\}+\mathbb{Z}$.  Let us write
$$A_{\mathrm{bad}}=(-\alpha+A_-) \cup (\alpha+A_+),$$
where $A_- \subseteq \mathbb{Z}_{\geq 1}$ and $A_+ \subseteq \mathbb{Z}_{\geq 0}$ and $|A_-|+|A_+|=\ell+1$.  Consider the auxiliary set $$B':=(A_1 \times \{0\}) \cup (A_- \times \{-1\}) \cup (A_+ \times \{1\}) \subseteq \mathbb{Z}^2,$$
which is obtained from $A$ by sending each element $i+j\alpha$ (with $j \in \{-1,0,1\}$ and $i \in \mathbb{Z}$) to $(i,j)$.  Observe that for each element $x \in \FS(A)$, one can uniquely determine the difference between the number of summands from $-\alpha+A_-$ and the number of summands from $\alpha+A_+$: There is a unique integer $-|A_-| \leq u \leq |A_+|$ such that $x \equiv u\alpha \pmod{1}$, and then the number of summands from $\alpha+A_+$ minus the number of summands from $-\alpha+A_-$ is exactly $u$.  It follows that two subset sums of $A$ coincide if and only if the corresponding subset sums in $B'$ coincide, so $|\FS(A)|=|\FS(B')|$.

Recall that replacing an element of $B'$ with its negative merely translates $\FS(B')$ and in particular leaves $|\FS(B')|$ unchanged.  With this in mind, let us flip the signs of $A_- \times \{-1\}$ to obtain the set
$$B:=(A_1 \times \{0\}) \cup ((A_+ \cup -A_-) \times \{1\}),$$
which satisfies $|\FS(B')|=|\FS(B)|$.  Notice that $|A_+ \cup -A_-|=|A_+|+|A_-|=\ell+1$, since if there were some $t \in A_+ \cap (-A_-)$, then $A$ would contain the elements $t+\alpha$ and $-t-\alpha$, one of which would have to be negative.  \cref{prop:calculation-for-tightness} again provides the desired contradiction. 
\end{proof}

\section{Further directions}\label{sec:open}

We conclude by discussing some further (mostly open) questions raised by our work.

\subsection{Extensions of \cref{thm:linear}}

Freiman's $3k-4$ Theorem says that if $A \subseteq \mathbb{R}$ is an $n$-element set with $|A+A| \leq 3n-4$, then $A$ is ``close'' to being an arithmetic progression, in the sense that it is contained in an arithmetic progression of length $|A+A|-n+1$.  The $3k-3$ Theorem provides a looser structural constraint on $n$-element sets $A$ with $|A+A| \leq 3n-3$ (roughly speaking, either $A$ is contained in a short arithmetic progression, or it is the union of two arithmetic progressions with the same common difference).  From here, the situation becomes more complicated but is still somewhat understood as long as $|A+A|$ is at most around $4|A|$; see \cite{EGM,JM} and the references therein.

Our \cref{thm:linear} provides a precise characterization of the $n$-element sets $A \subseteq \mathbb{R}_{>0}$ with $|\FS(A)| \leq \binom{n+1}{2}+1+(n-4)$, and we showed that our characterization fails if one replaces $n-4$ with $n-3$.  In contrast with the sumset setting described in the previous paragraph, however, here we are optimistic that for any fixed constant $C$, one should (with sufficient case analysis) be able to precisely characterize the $n$-element sets $A \subseteq \mathbb{R}_{>0}$ with $|\FS(A)| \leq \binom{n+1}{2}+1+(n+C)$.  We expect that such sets will still be ``close'' to homogeneous arithmetic progressions, but the allowed lengths of these arithmetic progressions will depend on the particular structure of the smallest $\Theta_C(1)$ elements of $A$.  The reason for our optimism is that the repeated addition of the subset sum operation introduces ``smoothing'' that tends to fill in gaps away from the very smallest and very largest subset sum values; the study of numerical semigroups may be relevant (see, e.g., \cite{mel,gs}).

If one allows $C$ to grow sufficiently quickly with $n$, then of course one can no longer expect $A$ to resemble a homogeneous arithmetic progression.  It would be interesting to determine, in some precise sense, the threshold at which this change of behavior occurs.

We also mention that one can pose the same questions for sets $A$ that are allowed to contain negative numbers as well as positive numbers (see \cite[Theorems 4 and 6]{N95}); the methods of \cref{thm:linear} should yield analogous results in this setting.

\subsection{Littlewood--Offord theory and extensions of \cref{thm: main inverse thm}}

Let us expand on our earlier discussion of Littlewood--Offord theory.  Recall that this area concerns the anticoncentration properties of the random variable $S:=\sum_{x \in B}x$, where $B$ is a uniformly random subset of a fixed $n$-element multiset $A \subseteq \mathbb{R}_{>0}$.  Early results of Littlewood and Offord \cite{LO}, Erd\H{o}s \cite{Erdos45}, Erd\H{o}s and Moser \cite{EM}, and S\'ark\"ozy and Szemer\'edi \cite{SS} showed that $\max_{y \in \RR} \mathbb{P}[S=y] \ll n^{-1/2}$ for all $n$-element multisets $A$, and that $\max_{y \in \RR} \mathbb{P}[S=y] \ll n^{-3/2}$ for all $n$-element sets $A$.  Tao and Vu~\cite{TV09} later showed that if $\max_{y \in \RR}\mathbb{P}[S=y] \geq n^{-O(1)}$, then almost all of $A$ is contained in a GAP of rank $O(1)$ and size $n^{O(1)}$.

In the setting of \cref{thm: main inverse thm}, the assumption $|\FS(A)| \leq Cn^2$ implies that $\max_{y \in \RR}\mathbb{P}[S=y] \geq C^{-1}n^{-2}$, and then the result of Tao and Vu tells us that most of $A$ is contained in a GAP of rank $O_C(1)$ and size $n^{O_C(1)}$.  Such a conclusion is unsuitable for our purposes, however, because the slackness in the size condition (compare the output of \cref{subsec:small-doubling}) would prevent us from running our later rank-reduction arguments.

Consider the more general problem of studying the structure of $n$-element sets $A \subseteq \mathbb{R}_{>0}$ with $|\FS(A)| \leq Cn^\kappa$ for fixed $\kappa>2$ and $C>0$.  As mentioned in the introduction, any description of such sets $A$ would have to be more complicated than what appears in \cref{thm: main inverse thm}.  One might hope to show that most of $A$ is contained in a small symmetric GAP of rank at most $\lfloor \kappa \rfloor-1$.  In this direction, the inverse Littlewood--Offord results of Tao and Vu identify a large subset of $A$ that is contained in a GAP of rank $O_{\kappa,C}(1)$ and size $n^{O_{\kappa,C}(1)}$, but we do not see any way to guarantee that the rank is in fact at most $\lfloor \kappa \rfloor-1$.  The most appealing instances of this open problem are $\kappa=2.01$ and $\kappa=3$.

\subsection{More on hyperplane constraints}

Recall the quantity 
$$f_d(n,m)=\min\{|\FS(A)|:A\in\Xi_d(n,m)\}$$ 
from \cref{sec:Rd}, where $\Xi_d(n,m)$ denotes the collection of $n$-element subsets $A \subseteq \mathbb R^d$ such that each proper subspace contains at most $m$ elements of $A$. We saw in \cref{thm:stability} that $f_d(n,m) \gg_{d,\varepsilon} n^{d+1}$ whenever $m \leq (1-\varepsilon) n$.  We now discuss the problem of finding matching upper-bound constructions, for various ranges of $m$.

The no-$k$-in-a-line problem provides one source of upper-bound constructions, due to the observation that if $B \subseteq \mathbb{R}^{d-1}$ is an $n$-element set such that each proper affine subspace contains at most $K$ elements of $B$, then $A:=B \times \{1\} \subseteq \mathbb{R}^d$ belongs to $\Xi_d(n,K)$.  If such a set $B$ is contained in $[0,N]^{d-1}$, then $\FS(A) \subseteq [0,nN]^{d-1} \times [0,n]$ has size $|\FS(A)| \ll_d n^d N^{d-1}$. 
Recent work of Grebennikov and Kwan~\cite{gk} establishes the existence of such sets $B$ for a wide range of parameters $n,K,N$: Their Theorem 1.3 shows that for any positive integer $d$ and any real $\delta>0$, if $N,K$ are positive integers satisfying $K^*\leq K \leq N^{d-2}$ (for a suitable constant $K^*=K^*(d,\delta)>0$), then there is a set $B \subseteq [1,N]^{d-1}$ of size $|B| \geq (1-\delta)KN$ such that each proper affine subspace contains at most $K$ elements of $B$.   In our context, applying this result with $\delta=1/2$ (for instance) and setting $n:=\lfloor KN/2 \rfloor$ gives
\begin{equation}\label{eq:affine-construction}
f_d(n,K) \ll_d n^{d} (n/K)^{d-1}=n^{2d-1} K^{1-d}
\end{equation}
whenever $K \leq (2n)^{(d-2)/(d-1)}$ is sufficiently large relative to $d$.  In particular, taking $K$ as large as possible leads to the upper bound 
$$f_d(n,\lfloor (2n)^{(d-2)/(d-1)} \rfloor) \ll_d n^{2d-1} \cdot n^{2-d}=n^{d+1}.$$
It follows that
$$f_d(n,m) \asymp_{d,\varepsilon} n^{d+1}$$
in the entire range $\lfloor (2n)^{(d-2)/(d-1)} \rfloor \leq m \leq (1-\varepsilon)n$ (the constant $2$ can be improved to $1+o(1)$ by sending $\delta \to 0$ slowly).

The bound \eqref{eq:affine-construction} becomes gradually worse as $K$ decreases but still gives $f_d(n,K) \ll_d n^{2d-1}$ whenever $K \geq K^*$.  Recent work of Ghosal, Goenka, and Keevash~\cite{ggk} demonstrates the existence of sets $B \subseteq [1,N]^{d-1}$ of size $|B| \gg_d N$ such that each proper affine subspace contains at most $d$ elements of $B$, so we in fact obtain $f_d(n,K) \ll_d n^{2d-1}$ whenever $K \geq d$.  For the interesting remaining case $K=d-1$, Ghosal, Goenka, and Keevash prove the existence of sets $B \subseteq [1,N]^{d-1}$ of size $|B| \gg_d N/(\log N)^{1/(d-1)}$ such that each proper affine subspace contains at most $d-1$ elements of $B$, whence $f_d(n,d-1) \ll_d n^{2d-1}(\log n)^{d/(d-1)}$. It is known (see, e.g., the discussion in \cite{gk}) that the logarithmic factor can be removed for $d=3$, which corresponds to the famous no-three-in-a-line problem.  This and \cref{thm:stability} together give $n^4 \ll f_3(n,2)\ll n^5$; in the absence of better constructions, we suspect that the upper bound is the truth.

\begin{conjecture}
We have $f_3(n,2)=\Theta(n^{5})$.
\end{conjecture} 

At the opposite extreme, Brickell and Saks~\cite[Theorem 1.4]{bs} have proven that $$f_d(n,n-1)=2^{d-1} f_1(n-d,n-d-1)$$
whenever $n \geq d+5$, and Nathanson~\cite[Theorem 4]{mel} has computed that $$f_1(k,k-1)=\lfloor (k+1)^2/4 \rfloor+1$$ (the extremal construction is an interval of integers centered at either $0$ or $1/2$, depending on the parity of $k$).

\subsection{Continuous analogues} 

In this subsection we describe a natural measurable analogue of \cref{thm:stability}.  Recall that this theorem says that if a finite set $A \subseteq \mathbb{R}^d$ is ``robustly'' $d$-dimensional, then $|\FS(A)| \gg |A|^{d+1}$.  One can ask about the optimal value of the implicit constant in the case where $A$ is a discretization (or, more precisely, sequence of increasingly fine discretizations) of a positive-measure subset of $\mathbb{R}^d$, which leads to the following problem. 

Fix a positive integer $d$, and let $\mu$ denote the $d$-dimensional Lebesgue measure.  Suppose that $A \subseteq \mathbb{R}^d$ is a Lebesgue-measurable set of finite positive measure such that $\int_A \|x\|_{\ell^1} \,d\mu(x)<\infty$ (the choice of the $\ell^1$-norm is not important since all norms on $\mathbb{R}^d$ are equivalent).  The measurable analogue of the set of subset sums of $A$ is
$$\SI(A):=\left\{ \int_B x \,d\mu(x): ~B \subseteq A \text{ measurable} \right\};$$
The integrability condition on $A$ ensures that these ``subset integrals'' $\int_B x \,dx$ all exist, and that the entire set $\SI(A)$ is bounded.  In fact, $\SI(A)$ is none other than a translate, by half of the first moment of $A$, of the well-studied \emph{centroid body} of $A$ from convex geometry: We have
$$\SI(A)=Z(A)+\frac{1}{2}\int_A x \,d\mu(x),$$
where $Z(A)$ denotes a dilation of the centroid body of $A$.  Lyapunov's Convexity Theorem (see \cite[Chapter~IX]{DU} or \cite{Art}) guarantees that $Z(A)$ is convex and, in particular, measurable.

Following the theme of this paper, we wish to bound $\mu(\SI(A))$ from below in terms of $\mu(A)$. For any $A \subseteq \mathbb{R}^d$ and $\lambda > 0$, we have $\SI(\lambda \cdot A) = \lambda^{d+1} \SI(A)$, so it suffices to consider the case $\mu(A) = 1$. Thus, we wish to characterize the sets $A \subseteq \mathbb{R}^d$ of measure $1$ that minimize $\mu(\SI(A))=\mu(Z(A))$, i.e., have the ``fewest possible'' distinct subset integrals.  The classical Busemann--Petty Inequality \cite{Petty} (see also \cite{Gardner}) says that if we further restrict $A$ to be convex, then the minimum is achieved if and only if $A$ differs by a set of measure $0$ from an ellipsoid (of measure $1$) centered at the origin.  It is routine to remove the convexity assumption.\footnote{One first deduces the same conclusion for the family of star-convex sets centered at the origin, and then allows arbitrary sets satisfying the above integrability condition.  Each reduction uses a ``redistribution of mass'' argument to constrain the structure of potential extremizers.} After a brief calculation, we arrive at the following clean statement.

\begin{proposition}
Let $d$ be a positive integer.  Let $A \subseteq \mathbb{R}^d$ be a set satisfying $\int_A \|x\|_{\ell^1} \,d\mu(x)<\infty$ and $0<\mu(A)<\infty$.  Let $\Gamma$ denote the gamma function.  Then 
$$\mu(\SI(A))\geq\left(\frac{\Gamma\!\left(\frac{d}{2}+1\right)}{(d+1)\sqrt{\pi}\,\Gamma\!\left(\frac{d+1}{2}\right)}\right)^d \cdot \mu(A)^{d+1},$$
with equality if and only if $A$ differs by a set of measure $0$ from an ellipsoid centered at the origin.
\end{proposition}

We mention that there has been some work on stability results for the Busemann--Petty Inequality \cite{Ivaki} and other similar isoperimetric inequalities \cite{Ivaki2, Nguyen}; it might be interesting to investigate discrete or additive-combinatorial consequences.

\section*{Acknowledgments}
Colin Defant was supported by a Benjamin Peirce Fellowship at Harvard University.  Noah Kravitz was supported in part by a NSF Mathematical Sciences Postdoctoral Research Fellowship under grant DMS-2501336.  We thank Ben Green for helpful conversations.


\begin{thebibliography}{99999999999}

\bibitem[Art90]{Art}
Z. Artstein, Yet another proof of the Lyapunov convexity theorem. \emph{Proc. Amer. Math. Soc.}, {\bf 108} (1990), 89--91. 

\bibitem[BS93]{bs} E. Brickell and M. Saks, The number of distinct subset sums of a finite set of vectors.  \emph{JCTA}, {\bf 63.2} (1993), 234--256.

\bibitem[CDK26]{CDK26} R. Carpenter, C. Defant, and N. Kravitz, On the number of permutation-twisted dot products.  Preprint \href{https://arxiv.org/abs/2601.15276}{arXiv:2601.15276v1} (2026).

\bibitem[DU77]{DU}
J. Diestel and J. J. Uhl Jr., Vector Measures, AMS Mathematical Surveys, vol. 15, (1977). 

\bibitem[EGM14]{EGM} S. Eberhard, B. Green, and F. Manners, Sets of integers with no large sum-free subset.  \emph{Ann. of Math.}, {\bf 180.2} (2014), 621--652. 

\bibitem[Erd45]{Erdos45}
P. Erd\H{o}s, On a lemma of Littlewood and Offord. \emph{Bull. Amer. Math. Soc.}, {\bf 51} (1945), 898--902. 

\bibitem[EM47]{EM} P. Erd\H{o}s and L. Moser, Elementary Problems and Solutions: Solutions: E736. \emph{Amer. Math. Monthly}, {\bf 54.4} (1947), 229--230. 


\bibitem[GHHLLP20]{GHHLLP} W. Gao, M. Huang, W. Hui, Y. Li, C. Liu, and J. Peng, Sums of sets of abelian group elements. \emph{J. Number Theory}, {\bf 208} (2020), 208--229.

\bibitem[Gar13]{Gardner} 
R. J. Gardner, Geometric Tomography, 2nd ed., Cambridge Univ. Press, 2013. 

\bibitem[GGK25]{ggk} A. Ghosal, R. Goenka, and P. Keevash, On subsets of lattice cubes avoiding affine and spherical degeneracies. Preprint \href{https://arxiv.org/abs/2509.06935}{arXiv:2509.06935v1} (2025).

\bibitem[GS23]{gs} A. Granville and G. Shakan, Effective results on the size and structure of sumsets.  \emph{Combinatorica}, {\bf 43.6} (2023), 1139--1178.

\bibitem[GK25]{gk} A. Grebennikov and M. Kwan, No-$(k+ 1)$-in-line problem for large constant $k$.  Preprint \href{https://arxiv.org/abs/2510.17743}{arXiv:2510.17743v1} (2025). 

\bibitem[vHK23]{vHK} P. van Hintum and P. Keevash, Sharp bounds for the Tao-Vu Discrete John's Theorem.  Preprint \href{https://arxiv.org/abs/2309.12386}{arXiv:2309.12386v1} (2023).

\bibitem[Iva14]{Ivaki2}
M. N. Ivaki, On the stability of the p-affine isoperimetric inequality. \emph{J. Geom. Anal.}, {\bf 24} (2014), 1898--1911. 

\bibitem[Iva16]{Ivaki}
M. N. Ivaki, The planar Busemann--Petty centroid inequality and its stability. \emph{Trans. Amer. Math. Soc.}, {\bf 368} (2016). 

\bibitem[JM26]{JM} Y. Jing and A. Mudgal, A structure theorem for sets with doubling $4+\delta$.  Preprint \href{https://arxiv.org/abs/2604.25893}{arXiv:2604.25893v1} (2026).

\bibitem[Lev24]{lev2} V. Lev, A nonlinear bound for the number of subsequence sums. \emph{European J. Combin.}, {\bf 118} (2024), 103907.

\bibitem[Lev99]{lev} V. Lev, The structure of multisets with a small number of subset sums. \emph{Ast\'erisque}, {\bf 258} (1999), 179--186.

\bibitem[LO43]{LO}
J. E. Littlewood and A. C. Offord, On the number of real roots of a random algebraic equation. III. \emph{Rec. Math. [Mat. Sbornik] N.S.} {\bf 12}, (1943). 277--286. 

\bibitem[Nat95]{N95} M. Nathanson,  Inverse theorems for subset sums.  \emph{Trans. Amer. Math. Soc.} {\bf 347.4} (1995), 1409--1418.

\bibitem[Nat72]{mel} M. Nathanson, Sums of finite sets of integers.  \emph{Amer. Math. Monthly}, {\bf 79.9} (1972), 1010--1012. 

\bibitem[Ngu16]{Nguyen}
V. H. Nguyen, New approach to the affine P\'olya--Szeg\H{o} principle and the stability version of the affine Sobolev inequality. \emph{Adv. Math.}, {\bf 302} (2016), 1080--1110. 

\bibitem[Pet61]{Petty}
C. M. Petty, Centroid surfaces. \emph{Pacific J. Math.}, {\bf 11} (1961), 1535--1547. 

\bibitem[SS65]{SS}  A. S\'ark\"ozy and E. Szemer\'edi, \"Uber ein Problem von Erd\H{o}s und Moser.  \emph{Acta Arithmetica}, {\bf 11.2}
(1965) 205--208. 

\bibitem[TV06]{TV} T. Tao and V. Vu, \emph{Additive Combinatorics}.  Cambridge Studies in Advanced Mathematics, vol. 105, Cambridge University Press (2006).

\bibitem[TV08]{TV08} T. Tao and V. Vu, John-type theorems for generalized arithmetic progressions and iterated sumsets. \emph{Adv. Math.} {\bf 219} (2008), no. 2, 428--449.

\bibitem[TV09]{TV09} T. Tao and V. Vu, Inverse Littlewood-Offord theorems and the condition number of random discrete matrices.  \emph{Ann. of Math.}, {\bf 169} (2009), 595--632.

\bibitem[Ung82]{U82} P. Ungar,  $2N$ non-collinear points determine at least $2N$ directions.  \emph{J. Combin. Theory Ser. A} {\bf 33} (1982), 343--347. 

\end{thebibliography}
\end{document}